\documentclass[12pt,leqno]{amsart}
\usepackage{color,amsmath,amsfonts,amssymb,amscd,amsthm,amsbsy,upref, siunitx}
\numberwithin{equation}{section}
\def\hangbox to #1 #2{\vskip3pt\hangindent #1\noindent \hbox to #1{#2}$\!\!$}

\newtheorem{thm}{Theorem}[section]

\newtheorem{lem}[thm]{Lemma}
\newtheorem{cor}[thm]{Corollary}

\newtheorem{prop}[thm]{Proposition}
\theoremstyle{definition}

\newtheorem{problem}[thm]{Problem}

\theoremstyle{remark}
\newtheorem{rem}[thm]{Remark}

\def\N{{\mathbb N}}
\def\R{{\mathbb R}}

\def\Z{{\mathbb Z}}

\def\sfrac#1#2{\kern.1em\raise.5ex\hbox{$#1$}
        \kern-.1em/\kern-.05em\lower.25ex\hbox{$#2$}}

\def\vp{\varepsilon}

\newcommand{\fw}{\text{\fw}}

\begin{document}
\allowdisplaybreaks
\title{A Schauder basis for $L_2$ consisting of non-negative functions}

\author{ Daniel Freeman}
\address{Department of Mathematics and Statistics\\
St Louis University\\
St Louis MO 63103  USA
} \email{daniel.freeman@slu.edu}

\author{ Alexander M. Powell}
\address{Department of Mathematics\\
Vanderbilt University\\
Nashville TN 37240  USA
} \email{alexander.m.powell@vanderbilt.edu}

\author{ Mitchell A. Taylor}
\address{Department of Mathematics\\
University of California, Berkeley\\
Berkeley CA 94720~USA
} \email{mitchelltaylor@berkeley.edu}

\thanks{
The first author was supported by grant 353293 from the Simons Foundation.}


\maketitle
 
 \begin{abstract}
We prove that $L_2(\R)$ contains a Schauder basis  of non-negative functions. Similarly, $L_p(\R)$ contains a Schauder basic sequence of non-negative functions such that $L_p(\R)$ embeds into the closed span of the sequence.  We prove as well that if $X$ is a separable Banach space with the bounded approximation property, then any set in $X$ with dense span contains a quasi-basis (Schauder frame) for $X$. Furthermore, if $X$ is a separable Banach lattice with a bibasis then any set in $X$ with dense span contains a u-frame.
\end{abstract}

 \section{Introduction}
Given $1\leq p<\infty$, we are interested in what coordinate systems can be formed for  $L_p(\R)$ using only non-negative functions.  The most desirable coordinate systems are unconditional, but this property is too strong to impose on this situation.  Indeed, for all $1\leq p<\infty$, $L_p(\R)$ does not have an unconditional Schauder basis or even unconditional quasi-basis (Schauder frame) consisting of non-negative functions \cite{PS}.  In particular, both the positive and negative parts of an unconditonal Schauder basis must have infinite weight \cite{NV}.
 
 When considering subspaces of $L_p(\R)$, it is clear that any normalized sequence of  non-negative functions with disjoint support will be 1-equivalent to the unit vector basis of $\ell_p$.  This trivial method is essentially the only way to build an unconditional Schauder basic sequence of non-negative functions in $L_p(\R)$, as every normalized unconditional Schauder basic sequence of non-negative functions in $L_p(\R)$ is equivalent to the unit vector basis for $\ell_p$ \cite{JS}.  Likewise, if $(f_j,g_j^*)_{j=1}^\infty$ is an unconditional quasi-basis  for a closed subspace $X$ of $L_p(\R)$ and $(f_j)_{j=1}^\infty$ is a sequence of non-negative functions then $X$ embeds into $\ell_p$ \cite{JS}. 

The results for coordinate systems formed by non-negative functions are very different when one allows for conditionality.  Indeed, for all $1\leq p<\infty$, $L_p(\R)$ has a Markushevich basis consisting of non-negative functions, and $L_p(\R)$ has a quasi-basis whose vectors consist of non-negative functions \cite{PS}.  
For the case of conditional Schauder bases, Johnson and Schechtman constructed a Schauder basis for  $L_1(\R)$ consisting of non-negative functions \cite{JS}.  Their construction relies heavily on the structure of $L_1$, and the problem on the existence of conditional Schauder bases for $L_p(\R)$ remained open for all $1<p<\infty$.  Our main result is to provide a construction for a Schauder basis of $L_2(\R)$ consisting of non-negative functions.   For the remaining cases $1<p<\infty$ with $p\neq 2$, we are not able to build a Schauder basis for the whole space.  However, we prove that for all $1<p<\infty$ there exists a Schauder basic sequence $(f_j)_{j=1}^\infty$ of non-negative functions in $L_p(\R)$ such that $L_p(\R)$ embeds into the closed span of $(f_j)_{j=1}^\infty$.

There are interesting comparisons between results on coordinate systems of non-negative functions for $L_p(\R)$ and results  on coordinate systems of translations of a single function.  As is the case for non-negative functions, there does not exist an unconditional Schauder basis for $L_p(\R)$ consisting of translations of a single function (\cite{OZ} for $p=2$, \cite{OSSZ} for $1<p\leq 4$, and \cite{FOSZ} for $4< p$).  On the other hand, for the range $2<p<\infty$ there does exist a sequence  $(f_j)_{j=1}^\infty$  of translations of a single function in $L_p(\R)$ and a sequence of functionals $(g^*_j)_{j=1}^\infty$ in $L_p(\R)^*$ such that $(f_j,g_j^*)_{j=1}^\infty$ is an unconditional Schauder frame for $L_p(\R)$ \cite{FOSZ}. The corresponding result for the range $1<p<2$ is unknown, but for $1<p\leq2$ the sequence of functionals $(g^*_j)_{j=1}^\infty$ in $L_p(\R)^*$ cannot be chosen to be semi-normalized \cite{BC}.
We take a unifying approach and prove that for all $1\leq p<\infty$, there exists a Schauder frame $(f_j,g_j^*)_{j=1}^\infty$ of $L_p(\R)$ such that $(f_j)_{j=1}^\infty$ is a sequence of translations of a single non-negative function.  We obtain this result by first proving that if $X$ is any separable Banach space with the bounded approximation property and  $D\subseteq X$ has dense span in $X$ then there exists a Schauder frame for $X$ whose vectors are elements of $D$. We extend this result further in the case that $X$ is a Banach lattice. Schauder frames give convergence in norm for partial sums, but in Banach lattices we can also require convergence in order for partial sums. We extend our theorem  in this direction to prove that if $X$ is a Banach lattice with a bibasis and $D\subseteq X$ has dense span then there exists a u-frame for $X$ whose vectors are elements of $D$.  We define Schauder frames, bibases and u-frames in Section \ref{S:frame}.

\section{A positive Schauder basis for $L_2(\R)$}\label{S:basis_2}

Given a separable infinite dimensional Banach space $X$, a sequence of vectors $(x_j)_{j=1}^\infty$ in $X$ is called a {\em Schauder basis} of $X$ if for all $x\in X$ there exists a unique sequence of scalars $(a_j)_{j=1}^\infty$ such that 
\begin{equation}\label{E:basis}
x=\sum_{j=1}^\infty a_j x_j.  
\end{equation}
A Schauder basis $(x_j)_{j=1}^\infty$ is called {\em unconditional } if the series in \eqref{E:basis} converges in every order.   If $(x_j)_{j=1}^\infty$ is a Schauder basis then there exists a unique sequence of bounded linear functionals {\em $(x_j^*)_{j=1}^\infty$ called the {\em }biorthogonal functionals of $(x_j)_{j=1}^\infty$ } such that $x_j^*(x_j)=1$ for all $j\in\N$ and $x_j^*(x_i)=0$ for all $j\neq i$. A sequence of vectors is called {\em basic} if it is a Schauder basis for its closed span.  A basic sequence $(x_j)$ is called $C$-basic for some constant $C>0$ if for all $m\leq n$ we have that
\begin{equation}\label{E:C_basis}
\left\|\sum_{j=1}^m a_j x_j\right\|\leq C \left\|\sum_{j=1}^n a_j x_j\right\|\quad\textrm {for all sequences of scalars }(a_j)_{j=1}^n. 
\end{equation}
It follows from the uniform boundedness principle that every basic sequence is $C$-basic for some constant $C$.  The least value $C$ such that a sequence $(x_j)$ is $C$-basic is called the {\em basis constant} of $(x_j)$.

 Question 9.1 in \cite{PS} asked if given $1\leq p<\infty$, does there exist a Schauder basis for $L_p(\R)$ consisting of non-negative functions?  This was recently solved for $L_1(\R)$ \cite{JS}, but all other cases remained open.
Our goal in this section is to give a procedure for creating a Schauder basis for $L_2(\R)$ formed of non-negative functions.  Will be using the terms positive and non-negative interchangeably as the set of non-negative functions in $L_p(\R)$ is the positive cone of $L_p(\R)$ when viewed as a Banach lattice.

There does not exist an unconditional positive Schauder basis for $L_p(\R)$ for any $1\leq p<\infty$ \cite{PS}. Thus, any positive Schauder basis we create must necessarily be conditional, and the property of conditionality will factor heavily into our construction.  The following lemma is our main tool, and it is based on  a classical construction for a conditional Schauder basis for $\ell_2$ (see for example pages 235-237 in \cite{AK}).

\begin{lem}\label{L:conditional}
Let $\vp>0$ and $1\geq c>0$.  There exists $N\in\N$ and a sequence $(x_j)_{j=1}^{2N}$ in $\ell_2(\Z_{2N})\oplus \ell_2(\Z_{2N})$ such that 
\begin{enumerate}
\item $(x_j)_{j=1}^{2N}$ is $(1+\vp)$-basic.
\item The orthogonal projection of $(0)_{j=1}^{2N}\oplus(\frac{1}{\sqrt{N}},\frac{c}{\sqrt{N}},\frac{1}{\sqrt{N}},\frac{c}{\sqrt{N}}...)_{j=1}^{2N}$ onto the span of $(x_j)_{j=1}^{2N}$ has norm at most $\vp$.
\item The distance from $(0)_{j=1}^{2N}\oplus(\frac{-c}{\sqrt{N}},\frac{1}{\sqrt{N}},\frac{-c}{\sqrt{N}},\frac{1}{\sqrt{N}}...)_{j=1}^{2N}$ to the span of $(x_j)_{j=1}^{2N}$ is at most $\vp$.
\end{enumerate}
\end{lem}

\begin{proof}

 Let $N\in\N$ and $(a_j)_{j=1}^N\subseteq (0,\infty)$ such that $\sum_{j=1}^N j a_j^2<\vp^2$ and $\sum_{j=1}^N a_j>\vp^{-2}c^{-2}$.   We prove that such a sequence exists later in Lemma \ref{L:sum_p}.   Consider the space $\ell_2(\Z_{2N})\oplus \ell_2(\Z_{2N})$.  We use the notation $\ell_2(\Z_{2N})$ instead of $\ell_2^{2N}$ because we will be making use of the cyclic structure of $\Z_{2N}$.   Let $T_1$ be the right shift operator on $\ell_2(\Z_{2N})\oplus \ell_2(\Z_{2N})$.  That is,
$$T_1(a_1,a_2,...,a_{2N})\oplus(b_1,b_2,...,b_{2N})=(a_{2N},a_1,a_2,...,a_{2N-1})\oplus(b_{2N},b_1,b_2,...,b_{2N-1})
$$
For $m\in\N$, we let $T_m=(T_1)^m$.  We let $(e_j)_{j=1}^{2N}$ be the unit vector basis of  $\ell_2(\Z_{2N})\oplus 0$ and  $(f_j)_{j=1}^{2N}$ be the unit vector basis of  $0\oplus\ell_2(\Z_{2N})$.  We let $x_1\in \ell_2(\Z_{2N})\oplus \ell_2(\Z_{2N})$ be the vector 
$x_1= e_1 + \sum_{j=1}^{N} a_{j} e_{2j}+ \sum_{j=1}^N \vp a_{j} f_{2j} $ and $x_2=e_2+\vp c f_1$.  For all $1\leq n< N$, we let $x_{2n+1}= T_{2n} x_1$ and $x_{2n+2}= T_{2n} x_2$.  That is,

\setlength{\arraycolsep}{1pt} 
\medmuskip = 1mu 
\begin{center}
$\begin{array}{c c c c c c c c c c c c c c c c c c c c c c c c c c c c c c c c c c c c c c}
x_1&=&(&1, &a_1,&0,&a_2,&0,&a_3,&...,&a_{N-1},&0,&a_N&)&\oplus&(&0, &\vp a_1,&0,&\vp a_2,&0,&...&)&\\
x_2&=&(&0,&1,&0,&0,&0,&0,&...&0,&0&0&)&\oplus&(&\vp c,&0,&0,&0,&0,&...&)&\\
x_3&=&(&0,&a_N,&1,&a_1,&0,&a_2,&...,&a_{N-2},&0,&a_{N-1}&)&\oplus&(&0,&\vp a_N,&0,&\vp a_1,&0,&...&)&\\
x_4&=&(&0,&0,&0,&1,&0,&0,&...&0,&0&0&)&\oplus&(&0,&0,&\vp c,&0,&0,&...&)&\\
x_5&=&(&0,&a_{N-1},&0,&a_N,&1,&a_1,&...,&a_{N-3},&0,&a_{N-2}&)&\oplus&(&0,&\vp a_{N-1},&0,&\vp a_N,&0,&...&)&\\
x_6&=&(&0,&0,&0,&0,&0,&1,&...&0,&0&0&)&\oplus&(&0,&0,&0,&0,&\vp c,&...&)&\\

&&&&&&&&&&\vdots&&&&&&&&&&&&\vdots\\
x_{2N-3}&=&(&0,&a_3,&0,&a_4,&0,&a_5,&...&a_1,&0,&a_2&)&\oplus&(&0,&\vp a_3,&0,&\vp a_4,&0,&...&)&\\
x_{2N-2}&=&(&0,&0,&0,&0,&0,&0,&...&1,&0,&0&)&\oplus&(&0,&0,&0,&0,&0,&...&)&\\
x_{2N-1}&=&(&0,&a_2,&0,&a_3,&0,&a_4,&...&a_N,&1,&a_1&)&\oplus&(&0,&\vp a_2,&0,&\vp a_3,&0,&...&)&\\
x_{2N}&=&(&0,&0,&0,&0,&0,&0,&...&0,&0,&1&)&\oplus&(&0,&0,&0,&0,&0,&...&)&\\
\end{array}$
\end{center}

Let $x=\sum_{j=1}^N \frac{1}{\sqrt{N}}f_{2j-1}+\sum_{j=1}^N\frac{c}{\sqrt{N}}f_{2j}$ and  $y=\sum_{j=1}^N \frac{c}{\sqrt{N}}f_{2j-1}+\sum_{j=1}^N\frac{-1}{\sqrt{N}}f_{2j}$
We will prove that this sequence $(x_j)_{j=1}^{2N}$ satisfies:

\begin{enumerate}
\item[(a)] $(x_j)_{j=1}^{2N}$ is $(1+4\vp)$-basic.
\item[(b)] The orthogonal projection of $x$ onto the span of $(x_j)_{j=1}^{2N}$ has norm at most $3 c \vp$.
\item[(c)] The distance from $y$ to the span of $(x_j)_{j=1}^{2N}$ is at most $\vp$.
\end{enumerate}

We first prove $(b)$.  We let  $Px$ be the orthogonal projection of $x$ onto the span of $(x_j)_{j=1}^{2N}$.   By symmetry, $Px$ will have the form $\sum_{j=1}^N a x_{2j-1}+\sum_{j=1}^N b x_{2j}$ for some $a,b\in \R$.  One can check that if $a=0$ then $\|Px\|=\vp c(1+\vp^2 c^2)^{-1/2}<3\vp c$.  We now assume that $a\neq0$.  Thus, 
$$\|Px\|=\frac{\langle x, \sum_{j=1}^N a x_{2j-1}+\sum_{j=1}^N b x_{2j}\rangle}{\|\sum_{j=1}^N a x_{2j-1}+\sum_{j=1}^N b x_{2j}\|}
=\max_{\beta\in\R}\frac{\langle x, \sum_{j=1}^N  x_{2j-1}+\sum_{j=1}^N \beta x_{2j}\rangle}{\|\sum_{j=1}^N x_{2j-1}+\sum_{j=1}^N \beta x_{2j}\|}
$$
By taking the derivative with respect to $\beta$, the maximum will be obtained when
\begin{equation}\label{E:critical}
\frac{d}{d\beta} \Big\langle x, \sum_{j=1}^N  x_{2j-1}+\sum_{j=1}^N \beta x_{2j}\Big\rangle{\Big\|\sum_{j=1}^N x_{2j-1}+\sum_{j=1}^N \beta x_{2j}\Big\|}
=\frac{d}{d\beta}{\Big\|\sum_{j=1}^N x_{2j-1}+\sum_{j=1}^N \beta x_{2j}\Big\|} \Big\langle x, \sum_{j=1}^N  x_{2j-1}+\sum_{j=1}^N \beta x_{2j}\Big\rangle.
\end{equation}

Let $A=\sum_{j=1}^N a_j$.  Then we get the following simplified expansion.
\begin{align*}
\sum_{j=1}^N x_{2j-1}+\sum_{j=1}^N \beta x_{2j}
&=\sum_{j=1}^N e_{2j-1}+\sum_{j=1}^{N}\Big(\beta+\sum_{i=1}^N a_i\Big)e_{2j}
+\sum_{j=1}^{N}\vp c\beta f_{2j-1}+\sum_{j=1}^N\Big(\vp\sum_{i=1}^N a_i\Big)f_{2j} \\
&=\sum_{j=1}^N  e_{2j-1}+\sum_{j=1}^{N}(\beta+A)e_{2j}
+\sum_{j=1}^{N}\vp c \beta f_{2j-1}+\sum_{j=1}^N \vp A f_{2j}\\
\end{align*}

This gives,
\begin{equation}\label{E:norm}
\Big\|\sum_{j=1}^N x_{2j-1}+\sum_{j=1}^N \beta x_{2j}\Big\|=\big(N+N(\beta+A)^2+N\vp^2 c^2\beta^2+N\vp^2A^2\big)^{1/2}
\end{equation}
\begin{equation}
\frac{d}{d\beta}\Big\|\sum_{j=1}^N x_{2j-1}+\sum_{j=1}^N \beta x_{2j}\Big\|=\big(N+N(\beta+A)^2+N\vp^2c^2\beta^2+N\vp^2A^2\big)^{-1/2}\big(N(\beta+A)+N\vp^2 c^2\beta\big)
\end{equation}
\begin{equation}\label{E:inner}
\Big\langle x,\sum_{j=1}^N x_{2j-1}+\sum_{j=1}^N \beta x_{2j}\Big\rangle= N^{1/2}\vp c\beta+N^{1/2}\vp c A
\end{equation}
\begin{equation}
\frac{d}{d\beta}\Big\langle x,\sum_{j=1}^N x_{2j-1}+\sum_{j=1}^N \beta x_{2j}\Big\rangle= N^{1/2}\vp c
\end{equation}
Substituting the above equalities into Equation \eqref{E:critical} gives that 
$$
N^{1/2}\vp c\big(N+N(\beta+A)^2+N\vp^2 c^2\beta^2+N\vp^2A^2\big)^{1/2}=\frac{(N^{1/2}\vp c\beta+N^{1/2}\vp c A)(N(\beta+A)+N\vp^2 c^2\beta)
}{\big(N+N(\beta+A)^2+N\vp^2 c^2\beta^2+N\vp^2A^2\big)^{1/2}}
$$
Multiplying both sides by the denominator and dividing by $N^{3/2}\vp c$ gives the following.
\begin{align*}
1+(\beta+A)^2+\vp^2 c^2\beta^2+\vp^2A^2&=(\beta+A)(\beta+A+\vp^2 c^2\beta)\\
1+(\beta+A)^2+\vp^2c^2\beta^2+\vp^2A^2&=(\beta+A)^2 + \vp^2c^2\beta^2 +\vp^2c^2\beta A\\
1+\vp^2A^2&=\vp^2c^2\beta A\\
\end{align*}
Thus, the critical point is at $\beta=\frac{1+\vp^2 A^2}{\vp^2 c^2 A}$.  Hence, $\sum_{j=1}^N x_{2j-1}+\sum_{j=1}^N\frac{1+\vp^2 A^2}{\vp^2 c^2A} x_{2j}$ will be a scalar multiple of the projection $Px$.  We now use \eqref{E:norm} to obtain a lower bound for the following.

\begin{align*}
\Big\|\sum_{j=1}^N x_{2j-1}+\sum_{j=1}^N \frac{1+\vp^2 A^2}{\vp^2c^2 A} x_{2j}\Big\|&>\Big\|\sum_{j=1}^N x_{2j-1}+\sum_{j=1}^N \frac{\vp^2 A^2}{\vp^2c^2 A} x_{2j}\Big\|\\
&=\Big\|\sum_{j=1}^N x_{2j-1}+\sum_{j=1}^N c^{-2}A x_{2j}\Big\|\\
&=(N+N(c^{-2}A+A)^2+N\vp^2c^2(c^{-2}A)^2+N\vp^2A^2)^{1/2}\quad\textrm{ by }\eqref{E:norm}\\
&>N^{1/2}c^{-2}A\qquad\textrm{ by the second term in the sum.}
\end{align*}
We now use \eqref{E:inner} to obtain an upper bound for the following.
\begin{align*}
\Big\langle x, \sum_{j=1}^N x_{2j-1}+\sum_{j=1}^N \frac{1+\vp^2 A^2}{\vp^2 c^2A} x_{2j}\Big\rangle&<\Big\langle x, \sum_{j=1}^N x_{2j-1}+\sum_{j=1}^N \frac{2\vp^2 A^2}{\vp^2 c^2A} x_{2j}\Big\rangle\\
&= \Big\langle x, \sum_{j=1}^N x_{2j-1}+\sum_{j=1}^N 2c^{-2}A x_{2j}\Big\rangle\\
&=N^{1/2}\vp c(2c^{-2}A)+N^{1/2}\vp c A\quad\textrm{ by }\eqref{E:inner}\\
&<3c^{-1}N^{1/2}\vp A
\end{align*}

We obtain an upper bound on $\|Px\|$ by
\begin{align*}
\|Px\|&=\frac{\langle x, \sum_{j=1}^N x_{2j-1}+\sum_{j=1}^N \frac{1+\vp^2 A^2}{\vp^2c^2 A} x_{2j}\rangle}{\|\sum_{j=1}^N x_{2j-1}+\sum_{j=1}^N \frac{1+\vp^2 A^2}{\vp^2c^2 A} x_{2j}\|}\\
&<\frac{3c^{-1}N^{1/2}\vp A}{c^{-2}N^{1/2}A}\\
&=3 c \vp
\end{align*}

This proves (b).  We will now prove (c).  

We have that
\begin{align*}
\Big\|\Big(\sum_{j=1}^N \frac{-1}{\vp A N^{1/2}}x_{2j-1}+\frac{1}{\vp N^{1/2}} x_{2j}\Big)-y\Big\|&=\Big\|\sum_{j=1}^N   \frac{-1}{\vp A N^{1/2}}e_{2j-1}\Big\|\\
&=\vp^{-1}A^{-1}\\
&<\vp   \qquad \textrm{ as }A=\sum_{j=1}^N a_j>\vp^{-2}.
\end{align*}

This proves that the distance from $y$ to the span of $(x_j)_{j=1}^{2N}$ is at most $\vp$ and hence we have proven $(c)$.

We now prove (a).  Let $0\leq M<N$ and $(b_j)_{j=1}^{2N}\in\ell_2(\Z_{2N})$.  We will first prove that $\|\sum_{j=1}^{2M+1} b_j x_j\|\leq(1+4\vp)\|\sum_{j=1}^{2N} b_j x_j\|$. 

The series $\sum_{j=1}^{2N} b_j x_j$ is expressed in terms of the basis $(e_j)_{j=1}^{2N}\cup(f_j)_{j=1}^{2N}$ by

\begin{equation}\label{E:full} 
\sum_{j=1}^{2N} b_j x_j= \sum_{j=1}^N b_{2j-1}e_{2j-1}+\sum_{j=1}^{N}\Big(b_{2j}+\sum_{i=0}^{N-1} b_{2i+1} a_{j-i}\Big)e_{2j}
+\sum_{j=1}^{N}\vp b_{2j}c f_{2j-1}+\sum_{j=1}^N\Big(\vp\sum_{i=0}^{N-1} b_{2i+1} a_{j-i}\Big)f_{2j}.
\end{equation}

The series $\sum_{j=1}^{2M+1} b_j x_j$ is expressed in terms of the basis $(e_j)_{j=1}^{2N}\cup(f_j)_{j=1}^{2N}$ by
\begin{equation}\label{E:part}
\sum_{j=1}^{2M+1} b_j x_j= \sum_{j=1}^{M+1} b_{2j-1}e_{2j-1}+y_{1,1}+
y_{1,2}+\sum_{j=1}^{M}\vp b_{2j}c f_{2j-1}+y_{2,1}+ y_{2,2}.
\end{equation}
Where, 
$$
y_{1,1}= \sum_{j=1}^{M}\Big(b_{2j}+\sum_{i=0}^{M} b_{2i+1} a_{j-i}\Big)e_{2j}\quad\textrm{ and }\quad
y_{1,2}=\sum_{j=M+1}^N\Big(\sum_{i=0}^{M} b_{2i+1} a_{j-i}\Big)e_{2j}
$$
$$
y_{2,1}= \sum_{j=1}^{M}\Big(\vp\sum_{i=0}^{M} b_{2i+1} a_{j-i}\Big)f_{2j}\quad\textrm{ and }\quad
y_{2,2}=\sum_{j=M+1}^N\Big(\vp\sum_{i=0}^{M} b_{2i+1} a_{j-i}\Big)f_{2j}
$$

Note that
\begin{equation}\label{E:compare2}
\Big\|\sum_{j=1}^{2N} b_j x_j\Big\|^2\geq \Big\|\sum_{j=1}^N b_{2j-1} e_{2j-1}\Big\|^2= \sum_{j=1}^{N}b_{2j-1}^2
\end{equation}

We first show that $\|y_{1,2}\|<\vp\|\sum_{j=1}^{2N} b_j x_j\|$.

\begin{align*}
\|y_{1,2}\|^2&=\Big\|\sum_{j=M+1}^N\Big(\sum_{i=0}^{M} b_{2i+1} a_{j-i}\Big)e_{2j}\Big\|^2\\
&=\sum_{j=M+1}^N \Big|\sum_{i=0}^{M} b_{2i+1} a_{j-i}\Big|^2\\
&\leq \sum_{j=M+1}^N \Big(\sum_{i=0}^{M} b_{2i+1}^2\Big)\Big(\sum_{i=0}^{M} a_{j-i}^2\Big)\quad\textrm{ by Cauchy-Schwartz}\\ 
&\leq \Big\|\sum_{j=1}^{2N} b_j x_j\Big\|^2\sum_{j=M+1}^N \sum_{i=0}^{M} a_{j-i}^2\quad\textrm{ by \eqref{E:compare2}}\\
&\leq\Big\|\sum_{j=1}^{2N} b_j x_j\Big\|^2\sum_{j=1}^N j a_j^2<\vp^2\Big\|\sum_{j=1}^{2N} b_j x_j\Big\|^2
\end{align*}
Thus we have that,
\begin{equation}\label{E:y12}
\|y_{1,2}\|<\vp\Big\|\sum_{j=1}^{2N} b_j x_j\Big\|.
\end{equation}
The same argument as above gives the following inequality.
\begin{equation}\label{E:y12same}
 \Big\| \sum_{j=1}^{M}\Big(\sum_{i=M+1}^{N-1} b_{2i+1} a_{j-i}\Big)e_{2j}\Big\| <\vp\Big\|\sum_{j=1}^{2N} b_j x_j\Big\|
\end{equation}
We can now estimate $\|y_{1,1}\|$.
\begin{align*}
\|y_{1,1}\|&=\Big\| \sum_{j=1}^{M}\Big(b_{2j}+\sum_{i=0}^{M} b_{2i+1} a_{j-i}\Big)e_{2j}\Big\|\\
&<\Big\|\sum_{j=1}^{M}\Big(b_{2j}+\sum_{i=0}^{M} b_{2i+1} a_{j-i}\Big)e_{2j}\Big\|-\Big\|\sum_{j=1}^{M}\Big(\sum_{i=M+1}^{N-1} b_{2i+1} a_{j-i}\Big)e_{2j}\Big\| +\vp\Big\|\sum_{j=1}^{2N} b_j x_j\Big\|\quad\textrm{ by }\eqref{E:y12same}\\
&\leq\Big\|\sum_{j=1}^{M}\Big(b_{2j}+\sum_{i=0}^{N-1} b_{2i+1} a_{j-i}\Big)e_{2j}\Big\|+\vp\Big\|\sum_{j=1}^{2N} b_j x_j\Big\|\\
&=\Big\|\sum_{j=1}^{M}\Big(b_{2j}+\sum_{i=1}^{N} b_{2j-2i-1} a_{i}\Big)e_{2j}\Big\|+\vp\Big\|\sum_{j=1}^{2N} b_j x_j\Big\|\\
\end{align*}
Thus, we have that
\begin{equation}\label{E:y11}
\|y_{1,1}\|<\Big\|\sum_{j=1}^{M}\Big(b_{2j}+\sum_{i=1}^{N} b_{2j-2i-1} a_{i}\Big)e_{2j}\Big\|+\vp\Big\|\sum_{j=1}^{2N} b_j x_j\Big\|
\end{equation}

The same technique for estimating $y_{1,1}$ and $y_{1,2}$ gives that 
\begin{equation}\label{E:y2}
\|y_{2,1}\|<\Big\|\sum_{j=1}^{M}\Big(\vp\sum_{i=1}^{N} b_{2j-2i-1} a_{i}\Big)f_{2j}\Big\|+\vp\Big\|\sum_{j=1}^{2N} b_j x_j\Big\|\quad
\textrm{ and }\quad \|y_{2,2}\|<\vp\Big\|\sum_{j=1}^{2N} b_j x_j\Big\|.
\end{equation}

We consider \eqref{E:part} with the inequalities \eqref{E:y12}, \eqref{E:y11}, and \eqref{E:y2} to get
\begin{align*}
\Big\|\sum_{j=1}^{2M+1} b_j x_j\Big\|&<  \Big\|\sum_{j=1}^{M+1} b_{2j-1}e_{2j-1}+\sum_{j=1}^{M}\Big(b_{2j}+\sum_{i=1}^N b_{2j-2i-1} a_i\Big)e_{2j}\\
&\qquad\qquad+\sum_{j=1}^{M}\vp b_{2j}c f_{2j-1}+ \sum_{j=1}^M\Big(\vp\sum_{i=1}^N b_{2j-2i-1} a_i\Big)f_{2j}\Big\|+4\vp \Big\|\sum_{j=1}^{2N} b_j x_j\Big\|\\
&\leq \Big\|\sum_{j=1}^{N} b_{2j-1}e_{2j-1}+\sum_{j=1}^{N}\Big(b_{2j}+\sum_{i=1}^N b_{2j-2i-1} a_i\Big)e_{2j}\\
& \qquad\qquad+\sum_{j=1}^{N}\vp b_{2j}c f_{2j-1}+ \sum_{j=1}^N\Big(\vp\sum_{i=1}^N b_{2j-2i-1} a_i\Big)f_{2j}\Big\|+4\vp \Big\|\sum_{j=1}^{2N} b_j x_j\Big\|\\
&= \Big\|\sum_{j=1}^{2N} b_j x_j\Big\|+4\vp \Big\|\sum_{j=1}^{2N} b_j x_j\Big\|\
\end{align*}

This proves for all $0\leq M< N$ that $\|\sum_{j=1}^{2M+1} b_j x_j\|\leq(1+4\vp)\|\sum_{j=1}^{2N} b_j x_j\|$.  The same argument proves that also $\|\sum_{j=1}^{2M} b_j x_j\|\leq(1+4\vp)\|\sum_{j=1}^{2N} b_j x_j\|$.  Thus, the sequence $(x_j)_{j=1}^{2N}$ has basic constant $(1+4\vp)$ and we have proven $(a)$.

\end{proof}

Before presenting our main theorem, we discuss the central idea behind our construction and its relation to the construction of Johnson and Schechtman \cite{JS}.   The conditional Schauder basis for $L_1(\R)$ constructed by Johnson and Schechtman can be formed inductively where at each step they break up a Haar vector $f$ into a positive part $f^+$ and a negative part $f^-$ then append a vector $2\cdot 1_{(n,n+1)}$ to both parts where $(n,n+1)$ is disjoint from the support of all vectors created so far in the induction process.  The vectors $f^+ +2\cdot1_{(n,n+1)}$ and $f^- +2\cdot1_{(n,n+1)}$ are then both positive vectors. One can then recover the vector $f$ by $f=(f^+ +2\cdot1_{(n,n+1)})-(f^- +2\cdot1_{(n,n+1)})$. Furthermore, the zero vector is the closest vector to $f^++f^-$ in the span of $f^+ +2\cdot 1_{(n,n+1)}$ and $f^- +2\cdot1_{(n,n+1)}$.
This idea can be used to build a Schauder basis for $L_1(\R)$, but it fails for $L_p(\R)$ for all $1<p<\infty$.  

Our procedure for constructing a positive Schauder basis for $L_2(\R)$ is also constructed inductively.  However, at each step instead of breaking up a vector into 2 pieces, we break it up into many pieces.  That is, given $\vp>0$ and $f\in L_2(\R)$ we choose a suitably large $N\in\N$, and then we break up the positive part of $f$ into $N$ pieces $(f^+_n)_{n=1}^N$ with the same distribution and the negative part of $f$ into $N$ pieces $(f^-_n)_{n=1}^N$ with the same distribution.  Here we mean that two functions $g,h:\R\rightarrow\R$ have {\em the same distribution} if for all $J\subseteq \R$ we have that $\lambda(g^{-1}(J))=\lambda(h^{-1}(J))$ where $\lambda$ is Lebesgue measure.
Given $(f^+_n)_{n=1}^N$ and $(f^-_n)_{n=1}^N$, we use Lemma \ref{L:conditional} to create a  positive highly conditional basic sequence  $(x_n)_{n=1}^{2N}$ with disjoint support from $f$ and append $(x_{2n-1})_{n=1}^N$ onto $(f^-_n)_{n=1}^N$ and append $(x_{2n})_{n=1}^N$ onto $(f^+_n)_{n=1}^N$.  The vectors $f^+_n+ x_{2n}$ and $f^-_n+x_{2n-1}$ are then both positive vectors for all $n\in\N$.
The conditionality of $(x_{n})_{n=1}^{2N}$ allows for $f$ to be within $\vp$ of $(\sum_{n=1}^N f^+_n+ x_{2n})-(\sum_{n=1}^{N}f^-_n+ x_{2n-1})$ and for the orthogonal projection of $f^+ +f^-$ onto  $span_{1\leq n\leq N}\{f^+_n+x_{2n},f^-_n+ x_{2n-1}\}$ to have norm smaller than $\vp$.

\begin{thm}\label{T:basisL2}
For all $\vp>0$, there exists a positive Schauder basis for $L_2(\R)$ with basis constant at most $1+\vp$.
\end{thm}
\begin{proof}
Let $0<\vp<1/2$ and  $\vp_j\searrow 0$ such that $\sum \vp_j<\vp$ and $\prod (1+\vp_j)<1+\vp$.  Let $(h_j)_{j=1}^\infty$ be a Schauder basis for $L_2(\R)$ which is an enumeration of the union of the Haar bases for $L_2([n,n+1])$ for all $n\in\Z$.  We assume that $h_1=1_{[0,1]}$.
We will inductively construct a sequence of nonnegative vectors $(z_j)_{j=1}^\infty$ and an increasing sequence of integers $(N_j)_{j=1}^\infty$ such that for all $n\in\N$,
\begin{enumerate}
\item[(a)] $z_n$ is piecewise constant.

\item[(b)] $(z_j)_{j=1}^{N_n}$ is $\prod_{j\leq n}(1+\vp_j)$ basic. 
\item[(c)] $dist(h_n,span_{j\leq N_n}(z_j))<\vp_n$.
\end{enumerate}

We first claim that $(z_j)_{j=1}^\infty$ will be a Schauder basis for $L_2(\R)$ with basis constant at most $1+\vp$.  Indeed, by (b) the sequence $(z_j)_{j=1}^\infty$ is $\prod(1+\vp_j)<(1+\vp)$ basic.  By (c) the span of $(z_j)_{j=1}^\infty$ contains a perturbation of an orthonormal basis and hence has dense span.  Thus all that remains is to construct $(z_j)$ by induction.

For the  base case we take $z_1=h_1=1_{[0,1]}$ and $N_1=1$.  Thus all three conditions are trivially satisfied.  Now let $k\in\N$ and assume that $(z_j)_{j=1}^{N_k}$ are given to satisfy the induction hypothesis.  Without loss of generality we may assume that $h_{k+1}$ is not contained in the span of $(z_j)_{j=1}^{N_k}$.  This is because if $h_{k+1}\in span_{j\leq N_k}(z_j)$ we could just take $N_{k+1}=N_k+1$ and $z_{N_{k+1}}$ to be the indicator function of an interval with support disjoint from the support of $z_j$ for all $1\leq j\leq N_k$. This would trivially satisfy (a), (b), and (c).  Thus, we may assume that $P_{(span_{j\leq N_k}(z_j))^{\perp}}h_{k+1}\neq 0$.  If $y\in L_2(\R)$ we write $y=y^+-y^-$ where $y^+$ and $y^-$ are non-negative and disjoint.
Let $y$ be a multiple of $P_{(span_{j\leq k}(z_j))^{\perp}}h_{k+1}$ such that $\|y^-\|=1$ and $c:=\|y^+\|\leq 1$.   Note that $y$ is piecewise constant as $h_{k+1}$ and $(z_j)_{j=1}^{N_k}$ are all piecewise constant.  If $c=0$ set $z_{k+1}=y^-$ and $N_{k+1}=N_{k}+1$, else we proceed as follows:

Let $\vp'>0$. By Lemma \ref{L:conditional} there exists $N\in\N$ and $(x_j)_{j=1}^{2N}$ in $\ell_2(\Z_{2N}\oplus\Z_{2N})_+$ such that 
\begin{enumerate}
\item $(x_j)_{j=1}^{2N}$ is $(1+\vp')$-basic.
\item The orthogonal projection of $(0,...,0)\oplus(\frac{1}{\sqrt{N}},\frac{c}{\sqrt{N}},...,\frac{1}{\sqrt{N}},\frac{c}{\sqrt{N}})$ onto the span of $(x_j)_{j=1}^{2N}$ has norm at most $\vp'$.
\item The distance from $(0,...,0)\oplus(\frac{c}{\sqrt{N}},\frac{-1}{\sqrt{N}},...,\frac{c}{\sqrt{N}},\frac{-1}{\sqrt{N}})$ to the span of $(x_j)_{j=1}^{2N}$ is at most $\vp'$.
\end{enumerate}

Let $X_k$ be the span of $y$ and $(z_j)_{j=1}^{N_k}$. Note that $X_k$ is a space of simple functions with finitely many discontinuities.  We claim that there exists a sequence of finite unions of intervals $(G_{j})_{j=1}^{2N}$ in $\R$ such that 
\begin{enumerate}
\item[(i)] The sequence $(G_{j})_{j=1}^{2N}$ is pairwise disjoint.
\item[(ii)] $\cup_{j=1}^N G_{2j-1}$ is the support of $y^+$ and $\cup_{j=1}^N G_{2j}$ is the support of $y^-$.  
\item[(iii)] For all $x\in X_k$, the sequence of functions $(x|_{G_{2j-1}})_{j=1}^N$ all have the same distribution.  
\item[(iv)] For all $x\in X_k$, the sequence of functions $(x|_{G_{2j}})_{j=1}^N$ all have the same distribution.
\end{enumerate}

 To prove this, we let $(E_j)_{j=1}^{M_1}$ be a partition of the support of $y^+$ into intervals such that for all $1\leq j\leq M_1$ both $y$ and $z_i$ are constant on $E_j$ for all $1\leq i\leq N_k$.
We know by (a) that such a partition exists.
 Likewise, let  $(F_j)_{j=1}^{M_0}$ be a partition of the support of $y^-$ into intervals such that for all $1\leq j\leq M_0$ both $y$ and $z_i$ are constant on $F_j$ for all $1\leq i\leq N_k$.  For all $1\leq j\leq M_1$ let $(E_{i,j})_{i=1}^{N}$ be a partition of $E_j$ into intervals of equal length, and for all $1\leq j\leq M_0$ let $(F_{i,j})_{i=1}^{N}$ be a partition of $F_j$ into intervals of equal length.  For all $1\leq i\leq N$ we let $G_{2i-1}=\cup_{j=1}^{M_1}E_{i,j}$ and let  $G_{2i}=\cup_{j=1}^{M_0}F_{i,j}$.  By construction, $(G_{i})_{i=1}^{2N}$ satisfies (i),(ii),(iii), and (iv).
 
 Let $(H_j)_{j=1}^{2N}$ be a sequence of unit length intervals in $\R$ with pairwise disjoint support which is disjoint from the support of $y$ and the support of $z_j$ for all $1\leq j\leq N_k$.
 We now define a map $\Psi:\ell_2(\Z_{2N}\oplus\Z_{2N})\rightarrow L_2(\R)$ by 
 $$\Psi(a_1,...,a_{2N},b_1,...,b_{2N})=\sum_{j=1}^{N} c^{-1}N^{1/2}b_{2j-1} 1_{G_{2j-1}}y^+ +\sum_{j=1}^{N} N^{1/2}b_{2j} 1_{ G_{2j}}y^-+\sum_{j=1}^{2N} a_j 1_{H_{j}}
 $$
By (i),(ii),(iii), and that $\|y^+\|=c$ we have that $\|1_{G_{2j-1}}y^+\|=cN^{-1/2}$ for all $1\leq j\leq N$.  Likewise, as $\|y^-\|=1$ we have that $\|1_{G_{2j}}y^-\|=N^{-1/2}$ for all $1\leq j\leq N$.  Thus, $\Psi$ is an isometric embedding and maps positive vectors in $\ell_2(\Z_{2N})\oplus\ell_2(\Z_{2N})$ to positive vectors in $L_2(\R)$.  We let $N_{k+1}=N_k+2N$ and let $z_{N_k+j}=\Psi(x_j)$ for all $1\leq j\leq 2N$.  As $y$ is piecewise constant, $H_i$ is an interval, and $G_i$ is a finite union of intervals for all $1\leq i\leq 2N$, we have that $z_j$ is piecewise constant for all $N_k< j\leq N_{k+1}$.  Thus we have satisfied (a).

Note that $\Psi((0,...,0)\oplus(\frac{c}{\sqrt{N}},\frac{-1}{\sqrt{N}},...,\frac{c}{\sqrt{N}},\frac{-1}{\sqrt{N}}))=y$, thus by (3) the distance from $y$ to the span of $(z_j)_{j=N_k+1}^{N_{k+1}}$ is at most $\vp'$ which proves (c) if $\vp'$ is small enough.

Let $x\in span_{j\leq N_k} z_j$.  Let $(e_j)_{j=1}^{2N}$ denote the unit vector basis for the second coordinate of $\ell_2(\Z_{2N})\oplus\ell_2(\Z_{2N})$.
Then by (iii), we have that $\langle \Psi(e_{2j-1}),x\rangle=\langle \Psi(e_{2i-1}),x\rangle$ for  all $1\leq i,j\leq N$, and by (iv) we have that $\langle \Psi(e_{2j}),x\rangle=\langle \Psi(e_{2i}),x\rangle$ for all $1\leq i,j\leq N$.   We have that $x$ is orthogonal to $y$ and $y=\frac{c}{\sqrt{N}} \Psi(e_1)-\frac{1}{\sqrt{N}}\Psi(e_2)+...+\frac{c}{\sqrt{N}}\Psi(e_{2N-1})-\frac{1}{\sqrt{N}}\Psi(e_{2N})$. Thus the orthogonal projection of $x$ onto $\Psi(\ell_2(\Z_{2N})\oplus\ell_2(\Z_{2N}))$ is a multiple of $\Psi(e_1)+c\Psi(e_2)+...+\Psi(e_{2N-1})+c\Psi(e_{2N})$.  Hence by (2) the orthogonal projection of $x$ onto $span_{N_k< j\leq N_{k+1}}z_j=span_{1\leq j\leq 2N}\Psi(x_j)$ has norm at most $2\vp'\|x\|$.   The sequence $(z_j)_{j=1}^{N_k}$ is $\prod_{j\leq k}(1+\vp_j)$ basic and $(z_j)_{j=N_k+1}^{N_{k+1}}$ is $(1+\vp')$ basic.  The inner product between a unit vector in $span_{j\leq N_k} z_j$ and a unit vector in $span_{N_k< j\leq N_{k+1}}z_j$ is at most $2\vp'$.  Thus, if $\vp'$ is small enough then $(z_j)_{j=1}^{N_{k+1}}$ is $\prod_{j\leq k+1}(1+\vp_j)$ basic which proves (b).  This completes the construction of $(z_j)$ by induction.

\end{proof}
\begin{rem}
Similar to \cite{JS}, one can use classification theorems to extend the above result to all separable $L_2(\mu)$. See, for example, \cite{LW} or Section 2.7 of \cite{MN}.  That is, if $L_2(\mu)$ is separable then for all $\vp>0$ there exists a positive Schauder basis for $L_2(\mu)$ with basis constant at most $1+\vp$.
\end{rem}

\section{A basic sequence in $L_p(\R)$ for $1<p<\infty$.}

Our method in Section \ref{S:basis_2} repeatedly makes use of orthogonal projections onto subspaces of $L_2(\R)$.  This prevents us from extending the construction to $L_p(\R)$ for $p\neq 2$.  However, we are able to obtain the result for large subspaces of $L_p(\R)$.  Indeed, for $\vp>0$ and $1<p<\infty$, we will construct a positive $(2+\vp)$-basic sequence $(z_j)_{j=1}^\infty$ in $L_p(\R)$ such that $L_p(\R)$ is isomorphic to a subspace of the closed span of $(z_j)_{j=1}^\infty$.

\begin{lem}\label{L:sum_p}
For all $\vp>0$ and $1<p<\infty$ there exists $N\in\N$ and $a_n\searrow0$ such that $\sum_{n=1}^N a_n>\vp^{-2}$ and $\sum_{n=1}^N(\sum_{j=n}^N a_j^q)^{p/q}<\vp^p$ where $1/p+1/q=1$.
\end{lem}
\begin{proof}
We consider the function $f:[1,\infty)\rightarrow\mathbb{R}$ given by $f(x)=((x+1)\ln(x+1))^{-1}$.  Then,
$$\int_{1}^\infty f(x)\, dx=\int_{1}^\infty ((x+1)\ln(x+1))^{-1}dx=\infty.
$$
We also have the following upper bound,
\begin{align*}
\int_{1}^\infty \left(\int_x^\infty f(t)^q\,dt\right)^{p/q}\!dx&=\int_{1}^\infty \left(\int_x^\infty ((t+1)\ln(t+1))^{-q}\,dt\right)^{p/q}\!dx\\
&\leq \int_{1}^\infty \left(\int_x^\infty (t+1)^{-q}\,dt\right)^{p/q} \ln(x+1)^{-p} \,dx\quad\textrm{ as }\ln(t+1)^{-q}\leq \ln(x+1)^{-q}\\
&= (q-1)^{-p/q}\int_{1}^\infty (x+1)^{(1-q)p/q} \ln(x+1)^{-p} \,dx\\
&= (q-1)^{-p/q}\int_{1}^\infty (x+1)^{-1} \ln(x+1)^{-p} \,dx\quad\textrm{ as } p^{-1}\!+q^{-1}=1\\
&=(q-1)^{-p/q}(p-1)^{-1} \ln(2)^{1-p}.
\end{align*}
As $f$ is a decreasing function, we have that $\sum_{n=1}^\infty f(n)=\infty$ and $\sum_{n=1}^\infty(\sum_{j=n}^\infty f(j)^q)^{p/q}<\infty$.  Hence, for all $\vp>0$ we may choose $N\in\N$ and $a_n\searrow0$ such that  $\sum_{n=1}^N a_n>\vp^{-2}$ and $\sum_{n=1}^N(\sum_{j=n}^N a_j^q)^{p/q}<\vp^p$.  In particular, for all $\vp>0$ we may choose 
$$a_n=\Big((n+2)\ln(n+2)\Big)^{-1}\Big((q-1)^{-p/q}(p-1)^{-1} \ln(2)^{1-p}\Big)^{-1/p}\vp,$$
and then choose $N\in\N$ such that $\sum_{n=1}^N a_n>\vp^{-2}$.

\end{proof}

The following lemma is an extension of Lemma \ref{L:conditional} to $\ell_p(\Z_{2N})\oplus \ell_p(\Z_{2N})$ where $1<p<\infty$.  In the previous section we constructed a positive Schauder basis for all of $L_2(\R)$ and this required a variable $0<c\leq 1$ in Lemma \ref{L:conditional}.  For $p\neq 2$, we will only be constructing a positive Schauder basis for a subspace of $L_p(\R)$, and for this reason we will no longer need the variable $c$.

\begin{lem}\label{L:conditional_p}
Let $\vp>0$ and $1<p,q<\infty$ with $1/p+1/q=1$.  There exists $N\in\N$ and a sequence $(x_j)_{j=1}^{2N}$ in $\ell_p(\Z_{2N})\oplus \ell_p(\Z_{2N})$ such that 
\begin{enumerate}
\item $(x_j)_{j=1}^{2N}$ is $(1+\vp)$-basic.
\item If $f^*=(0)_{j=1}^{2N}\oplus(N^{-1/q})_{j=1}^{2N} \in \ell_q(\Z_{2N})\oplus \ell_q(\Z_{2N})$ then $|f^*(x)|\leq \vp\|x\|$  for all $x$ in the span of $(x_j)_{j=1}^{2N}$.
\item The distance from $(0)_{j=1}^{2N}\oplus((-1)^j N^{-1/p})$ to the span of $(x_j)_{j=1}^{2N}$ is at most $\vp$.
\end{enumerate}

\end{lem}

\begin{proof}

By Lemma \ref{L:sum_p}, there exists $N\in\N$ and $(a_j)_{j=1}^N\subseteq (0,\infty)$ such that 
\begin{equation}\label{E:a_p}
\sum_{n=1}^N a_n>\vp^{-2}\quad\textrm{ and }\quad\sum_{n=1}^N\Big(\sum_{j=n}^N a_j^q\Big)^{p/q}<\vp^p.
\end{equation}
  
      Consider the space $\ell_p(\Z_{2N}\oplus \Z_{2N})$.   Let $T_1$ be the cyclic right shift operator on $\ell_p(\Z_{2N}\oplus \Z_{2N})$.  That is,
$$T_1(a_1,a_2,...,a_{2N})\oplus(b_1,b_2,...,b_{2N})=(a_{2N},a_1,a_2,...,a_{2N-1})\oplus(b_{2N},b_1,b_2,...,b_{2N-1}).
$$
For $m\in\N$, we let $T_m=(T_1)^m$.  We let $(e_j)_{j=1}^{2N}$ be the unit vector basis of  $\ell_p(\Z_{2N}\oplus 0)$ and  $(f_j)_{j=1}^{2N}$ be the unit vector basis of  $\ell_p(0\oplus\Z_{2N})$.  We denote $(e_j^*)_{j=1}^{2N}$ and $(f_j^*)_{j=1}^{2N}$ to be the biorthogonal functionals to  $(e_j)_{j=1}^{2N}$  and $(f_j)_{j=1}^{2N}$. We let $x_1\in \ell_p(\Z_{2N})\oplus \ell_p(\Z_{2N})$ be the vector 
$x_1= e_1 + \sum_{j=1}^{N} a_{j} e_{2j}+ \sum_{j=1}^N \vp a_{j} f_{2j} $ and $x_2=e_2+\vp  f_1$.  For all $1\leq  n< N$, we let $x_{2n+1}= T_{2n} x_1$ and $x_{2n+2}= T_{2n} x_2$.  That is,

\setlength{\arraycolsep}{1pt} 
\medmuskip = 1mu 
\begin{center}
$\begin{array}{c c c c c c c c c c c c c c c c c c c c c c c c c c c c c c c c c c c c c c}
x_1&=&(&1, &a_1,&0,&a_2,&0,&a_3,&...,&a_{N-1},&0,&a_N&)&\oplus&(&0, &\vp a_1,&0,&\vp a_2,&0,&...&)&\\
x_2&=&(&0,&1,&0,&0,&0,&0,&...&0,&0&0&)&\oplus&(&\vp,&0,&0,&0,&0,&...&)&\\
x_3&=&(&0,&a_N,&1,&a_1,&0,&a_2,&...,&a_{N-2},&0,&a_{N-1}&)&\oplus&(&0,&\vp a_N,&0,&\vp a_1,&0,&...&)&\\
x_4&=&(&0,&0,&0,&1,&0,&0,&...&0,&0&0&)&\oplus&(&0,&0,&\vp ,&0,&0,&...&)&\\
x_5&=&(&0,&a_{N-1},&0,&a_N,&1,&a_1,&...,&a_{N-3},&0,&a_{N-2}&)&\oplus&(&0,&\vp a_{N-1},&0,&\vp a_N,&0,&...&)&\\
x_6&=&(&0,&0,&0,&0,&0,&1,&...&0,&0&0&)&\oplus&(&0,&0,&0,&0,&\vp,&...&)&\\

&&&&&&&&&&\vdots&&&&&&&&&&&&\vdots\\
x_{2N-3}&=&(&0,&a_3,&0,&a_4,&0,&a_5,&...&a_1,&0,&a_2&)&\oplus&(&0,&\vp a_3,&0,&\vp a_4,&0,&...&)&\\
x_{2N-2}&=&(&0,&0,&0,&0,&0,&0,&...&1,&0,&0&)&\oplus&(&0,&0,&0,&0,&0,&...&)&\\
x_{2N-1}&=&(&0,&a_2,&0,&a_3,&0,&a_4,&...&a_N,&1,&a_1&)&\oplus&(&0,&\vp a_2,&0,&\vp a_3,&0,&...&)&\\
x_{2N}&=&(&0,&0,&0,&0,&0,&0,&...&0,&0,&1&)&\oplus&(&0,&0,&0,&0,&0,&...&)&\\
\end{array}$
\end{center}

Let $f^*=\sum_{j=1}^{2N} N^{-1/q} f^*_{j}$ and  $y=\sum_{j=1}^{2N} (-1)^jN^{-1/p}f_{j}$
We will prove that the sequence $(x_j)_{j=1}^{2N}$ satisfies:

\begin{enumerate}
\item[(a)] $(x_j)_{j=1}^{2N}$ is $(1+4\vp)$-basic.
\item[(b)] $f^*(z)\leq\vp\|z\|$ for all $z$ in the span of $(x_j)_{j=1}^{2N}$.
\item[(c)] The distance from $y$ to the span of $(x_j)_{j=1}^{2N}$ is at most $\vp$.
\end{enumerate}

We first prove $(b)$.  As the unit ball of $\ell_p(\Z_{2N})\oplus_p \ell_p(\Z_{2N})$ is strictly convex, there exists a unique unit norm vector $z$ in the span of $(x_j)_{j=1}^{2N}$  so that $f^*(z)$ is maximal.
  By symmetry, $z$ will have the form $\sum_{j=1}^N a x_{2j-1}+\sum_{j=1}^N b x_{2j}$ for some $a,b\in \R$.  One can check that if $a=0$ then $f^*(z)=\vp(1+\vp^p)^{-1/p}<\vp$.
We now assume that $a\neq0$.
  Thus, 
$$f^*(z)=\frac{f^* (\sum_{j=1}^N a x_{2j-1}+\sum_{j=1}^N b x_{2j})}{\|\sum_{j=1}^N a x_{2j-1}+\sum_{j=1}^N b x_{2j}\|}
=\max_{\beta\in\R}\frac{| f^* (\sum_{j=1}^N  x_{2j-1}+\sum_{j=1}^N \beta x_{2j})|}{\|\sum_{j=1}^N x_{2j-1}+\sum_{j=1}^N \beta x_{2j}\|}
$$

Let $A=\sum_{j=1}^N a_j$.  Then we get the following simplified expansion.
\begin{align*}
\sum_{j=1}^N x_{2j-1}+\sum_{j=1}^N \beta x_{2j}
&=\sum_{j=1}^N e_{2j-1}+\sum_{j=1}^{N}\big(\beta+\sum_{i=1}^N a_i\big)e_{2j}
+\sum_{j=1}^{N}\vp \beta f_{2j-1}+\sum_{j=1}^N\big(\vp\sum_{i=1}^N a_i\big)f_{2j} \\
&=\sum_{j=1}^N  e_{2j-1}+\sum_{j=1}^{N}(\beta+A)e_{2j}
+\sum_{j=1}^{N}\vp  \beta f_{2j-1}+\sum_{j=1}^N \vp A f_{2j}\\
\end{align*}

This gives the following two equalities,
\begin{equation}\label{E:norm_p}
\Big\|\sum_{j=1}^N x_{2j-1}+\sum_{j=1}^N \beta x_{2j}\Big\|=\big(N+N|\beta+A|^p+N\vp^p |\beta|^p+N\vp^p A^p\big)^{1/p},
\end{equation}
\begin{equation}\label{E:inner_p}
f^*\Big(\sum_{j=1}^N x_{2j-1}+\sum_{j=1}^N \beta x_{2j}\Big)= N^{1/p}\vp \beta+N^{1/p}\vp  A.
\end{equation}

Let $\beta\in\R$ such that 
$$f^*(z)=\frac{|f^*( \sum_{j=1}^N  x_{2j-1}+\sum_{j=1}^N \beta x_{2j})|}{\|\sum_{j=1}^N x_{2j-1}+\sum_{j=1}^N \beta x_{2j}\|}.$$
For  $\lambda:=\beta/ A$, we have the following two equalities. 
$$
\Big\|\sum_{j=1}^N x_{2j-1}+\sum_{j=1}^N \beta x_{2j}\Big\|=\big(N+N|\lambda A+A|^p+N\vp^p (|\lambda |A)^p+N\vp^p A^p\big)^{1/p}
> \big(N|\lambda A+A|^p\big)^{1/p}=|1+\lambda| A N^{1/p}
$$
$$
f^*\Big(\sum_{j=1}^N x_{2j-1}+\sum_{j=1}^N \beta x_{2j}\Big)= N^{1/p}\vp \lambda A+N^{1/p}\vp  A
= \vp (1+\lambda) A N^{1/p}
$$

If $\lambda=-1$ then by the above equality we would have $ f^*(\sum_{j=1}^N x_{2j-1}+\sum_{j=1}^N \beta x_{2j})= 0$.   Otherwise, we have that,
$$|f^*(z)|< \vp |1+\lambda| A N^{1/p}/(|1+\lambda| A N^{1/p})=\vp
$$
Thus, we have proven (b).  We will now prove (c).  

Recall that $y=\sum_{j=1}^{2N} (-1)^jN^{-1/p}f_{j}$.  We have that
\begin{align*}
\Big\|\Big(\sum_{j=1}^N \frac{1}{\vp A N^{1/p}}x_{2j-1}-\frac{1}{\vp N^{1/p}} x_{2j}\Big)-y\Big\|&=\Big\|\sum_{j=1}^N   \frac{1}{\vp A N^{1/p}}e_{2j-1}\Big\|\\
&=\vp^{-1}A^{-1}\\
&<\vp   \qquad \textrm{ as }A=\sum_{j=1}^N a_j>\vp^{-2}.
\end{align*}

This proves that the distance from $y$ to the span of $(x_j)_{j=1}^{2N}$ is at most $\vp$ and hence we have proven $(c)$.

We now prove $(a)$.
Let $0\leq M<N$ and $(b_j)_{j=1}^{2N}\in\ell_p(\Z_{2N})$.  We will prove  that $\|\sum_{j=1}^{2M+1} b_j x_j\|\leq(1+4\vp)\|\sum_{j=1}^{2N} b_j x_j\|$.  

The series $\sum_{j=1}^{2N} b_j x_j$ is expressed in terms of the basis $(e_j)_{j=1}^{2N}\cup(f_j)_{j=1}^{2N}$ by

\begin{equation}\label{E:full_p} 
\sum_{j=1}^{2N} b_j x_j= \sum_{j=1}^N b_{2j-1}e_{2j-1}+\sum_{j=1}^{N}\Big(b_{2j}+\sum_{i=0}^{N-1} b_{2i+1} a_{j-i}\Big)e_{2j}
+\sum_{j=1}^{N}\vp b_{2j} f_{2j-1}+\sum_{j=1}^N\Big(\vp\sum_{i=0}^{N-1} b_{2i+1} a_{j-i}\Big)f_{2j}.
\end{equation}

The series $\sum_{j=1}^{2M+1} b_j x_j$ is expressed in terms of the basis $(e_j)_{j=1}^{2N}\cup(f_j)_{j=1}^{2N}$ by
\begin{equation}\label{E:part_p}
\sum_{j=1}^{2M+1} b_j x_j= \sum_{j=1}^{M+1} b_{2j-1}e_{2j-1}+y_{1,1}+
y_{1,2}+\sum_{j=1}^{M}\vp b_{2j} f_{2j-1}+y_{2,1}+ y_{2,2}.
\end{equation}
Where, 
$$
y_{1,1}= \sum_{j=1}^{M}\big(b_{2j}+\sum_{i=0}^{M} b_{2i+1} a_{j-i}\big)e_{2j}\quad\textrm{ and }\quad
y_{1,2}=\sum_{j=M+1}^N\big(\sum_{i=0}^{M} b_{2i+1} a_{j-i}\big)e_{2j}
$$
$$
y_{2,1}= \sum_{j=1}^{M}\big(\vp\sum_{i=0}^{M} b_{2i+1} a_{j-i}\big)f_{2j}\quad\textrm{ and }\quad
y_{2,2}=\sum_{j=M+1}^N\big(\vp\sum_{i=0}^{M} b_{2i+1} a_{j-i}\big)f_{2j}
$$

Note that
\begin{equation}\label{E:compare2_p}
\Big\|\sum_{j=1}^{2N} b_j x_j\Big\|^p\geq \Big\|\sum_{j=1}^N b_{2j-1} e_{2j-1}\Big\|^p= \sum_{j=1}^{N}|b_{2j-1}|^p
\end{equation}

We first show that $\|y_{1,2}\|<\vp\|\sum_{j=1}^{2N} b_j x_j\|$.

\begin{align*}
\|y_{1,2}\|^p&=\Big\|\sum_{j=M+1}^N\Big(\sum_{i=0}^{M} b_{2i+1} a_{j-i}\Big)e_{2j}\Big\|^p\\
&=\sum_{j=M+1}^N \left|\sum_{i=0}^{M} b_{2i+1} a_{j-i}\right|^p\\
&\leq \sum_{j=M+1}^N \Big(\sum_{i=0}^{M} |b_{2i+1}|^p\Big)\Big(\sum_{i=0}^{M} a_{j-i}^q\Big)^{p/q}\quad\textrm{ by H\"{o}lder's Inequality,}\\ 
&\leq \Big\|\sum_{j=1}^{2N} b_j x_j\Big\|^p\sum_{j=M+1}^N \Big(\sum_{i=0}^{M} a_{j-i}^q\Big)^{p/q}\quad\textrm{ by \eqref{E:compare2_p}},\\
&\leq \Big\|\sum_{j=1}^{2N} b_j x_j\Big\|^p\sum_{j=1}^N \Big(\sum_{i=j}^{N} a_{i}^q\Big)^{p/q}\\
&< \Big\|\sum_{j=1}^{2N} b_j x_j\Big\|^p \vp^p \qquad\textrm{ by \eqref{E:a_p}}
\end{align*}
Thus we have that,
\begin{equation}\label{E:y12_p}
\|y_{1,2}\|<\vp\Big\|\sum_{j=1}^{2N} b_j x_j\Big\|.
\end{equation}
The same argument as above gives the following inequality.
\begin{equation}\label{E:y12same_p}
 \Big\| \sum_{j=1}^{M}\Big(\sum_{i=M+1}^{N-1} b_{2i+1} a_{j-i}\Big)e_{2j}\Big\| <\vp\Big\|\sum_{j=1}^{2N} b_j x_j\Big\|
\end{equation}
We can now estimate $\|y_{1,1}\|$.
\begin{align*}
\|y_{1,1}\|&=\Big\| \sum_{j=1}^{M}\Big(b_{2j}+\sum_{i=0}^{M} b_{2i+1} a_{j-i}\Big)e_{2j}\Big\|\\
&<\Big\|\sum_{j=1}^{M}\Big(b_{2j}+\sum_{i=0}^{M} b_{2i+1} a_{j-i}\Big)e_{2j}\Big\|-\Big\|\sum_{j=1}^{M}\Big(\sum_{i=M+1}^{N-1} b_{2i+1} a_{j-i}\Big)e_{2j}\Big\| +\vp\Big\|\sum_{j=1}^{2N} b_j x_j\Big\|\quad\textrm{ by }\eqref{E:y12same_p}\\
&\leq\Big\|\sum_{j=1}^{M}\Big(b_{2j}+\sum_{i=0}^{N-1} b_{2i+1} a_{j-i}\Big)e_{2j}\Big\|+\vp\Big\|\sum_{j=1}^{2N} b_j x_j\Big\|\\
&=\Big\|\sum_{j=1}^{M}\Big(b_{2j}+\sum_{i=1}^{N} b_{2j-2i-1} a_{i}\Big)e_{2j}\Big\|+\vp\Big\|\sum_{j=1}^{2N} b_j x_j\Big\|\\
\end{align*}
Thus, we have that
\begin{equation}\label{E:y11_p}
\|y_{1,1}\|<\Big\|\sum_{j=1}^{M}\Big(b_{2j}+\sum_{i=1}^{N} b_{2j-2i-1} a_{i}\Big)e_{2j}\Big\|+\vp\Big\|\sum_{j=1}^{2N} b_j x_j\Big\|
\end{equation}

The same technique for estimating $y_{1,1}$ and $y_{1,2}$ gives that 
\begin{equation}\label{E:y2_p}
\|y_{2,1}\|<\Big\|\sum_{j=1}^{M}\Big(\vp\sum_{i=1}^{N} b_{2j-2i-1} a_{i}\Big)f_{2j}\Big\|+\vp\Big\|\sum_{j=1}^{2N} b_j x_j\Big\|\quad
\textrm{ and }\quad \|y_{2,2}\|<\vp\Big\|\sum_{j=1}^{2N} b_j x_j\Big\|.
\end{equation}

We consider \eqref{E:part_p} with the inequalities \eqref{E:y12_p}, \eqref{E:y11_p}, and \eqref{E:y2_p} to get
\begin{align*}
\Big\|\sum_{j=1}^{2M+1} b_j x_j\Big\|&<  \Big\|\sum_{j=1}^{M+1} b_{2j-1}e_{2j-1}+\sum_{j=1}^{M}\Big(b_{2j}+\sum_{i=1}^N b_{2j-2i-1} a_i\Big)e_{2j}\\
&\quad\qquad+\sum_{j=1}^{M}\vp b_{2j} f_{2j-1}+ \sum_{j=1}^M\Big(\vp\sum_{i=1}^N b_{2j-2i-1} a_i\Big)f_{2j}\Big\|+4\vp \Big\|\sum_{j=1}^{2N} b_j x_j\Big\|\\
&\leq \Big\|\sum_{j=1}^{N} b_{2j-1}e_{2j-1}+\sum_{j=1}^{N}\Big(b_{2j}+\sum_{i=1}^N b_{2j-2i-1} a_i\Big)e_{2j}\\
& \quad\qquad+\sum_{j=1}^{N}\vp b_{2j} f_{2j-1}+ \sum_{j=1}^N\Big(\vp\sum_{i=1}^N b_{2j-2i-1} a_i\Big)f_{2j}\Big\|+4\vp \Big\|\sum_{j=1}^{2N} b_j x_j\Big\|\\
&= \Big\|\sum_{j=1}^{2N} b_j x_j\Big\|+4\vp \Big\|\sum_{j=1}^{2N} b_j x_j\Big\|\
\end{align*}

This proves for all $0\leq M< N$ that $\|\sum_{j=1}^{2M+1} b_j x_j\|\leq(1+4\vp)\|\sum_{j=1}^{2N} b_j x_j\|$.  The same argument proves that also $\|\sum_{j=1}^{2M} b_j x_j\|\leq(1+4\vp)\|\sum_{j=1}^{2N} b_j x_j\|$.  Thus, the sequence $(x_j)_{j=1}^{2N}$ has basic constant $(1+4\vp)$ and we have proven $(a)$.
\end{proof}

We now show how the conditional positive basic sequence constructed in Lemma \ref{L:conditional_p} can be inductively used to build a basic sequence in $L_p(\R)$.
We will construct a positive basic sequence in $L_p(\R)$ which contains a perturbation of a Haar type system in $L_p([0,1])$.  Recall that a sequence of vectors $(g_j)_{j=0}^\infty$ in $L_p([0,1])$ is called a {\em Haar type system } if there is a sequence of partitions $(\{E_{j,n}\}_{j=0}^{2^n-1})_{n=0}^\infty$ of $[0,1]$ such that  $E_{0,0}=[0,1]$ and $g_0=1_{[0,1]}$ and for all $n\in \N$ and $0\leq j\leq 2^{n-1}-1$ we have that $\{E_{2j,n},E_{2j+1,n}\}$ is a partition of $E_{j,n-1}$ with $\lambda(E_{2j,n})=\lambda(E_{2j+1,n})=2^{-n}$ and  $g_{2^{n-1}+j}=2^{(n-1)/p}(1_{E_{{2j},n}}-1_{E_{{2j+1},n}})$.  Note that the  Haar basis for $L_p([0,1])$ is a Haar type system, and every Haar type system in $L_p([0,1])$ is 1-equivalent to the Haar basis.  Thus, if $(g_j)_{j=0}^\infty$ is a Haar type system in $L_p([0,1])$ then the closed span of $(g_j)_{j=0}^\infty$ is isometric to $L_p([0,1])$.  We will denote the usual Haar basis for $L_p([0,1])$ by $(h_j)_{j=0}^\infty$, and denote its dual sequence by $(h^*_j)_{j=0}^\infty$ (which is just the Haar basis for $L_q([0,1])$ for $1/p+1/q=1$.)

\begin{thm}\label{T:basisLp}
For all $1<p<\infty$, there exists a positive Schauder basic sequence $(z_j)_{j=0}^\infty$ in $L_p(\R)$  such that $L_p(\R)$ is isomorphic to a subspace of the closed span of $(z_j)_{j=0}^\infty$.
\end{thm}
\begin{proof}
Let $0<\vp<1$ and  $\vp_j\searrow 0$ such that $\sum 2\vp_j<\vp$ and $\prod (1+\vp_j)<1+\vp$. 
We will inductively construct a sequence of non-negative vectors $(z_j)_{j=0}^\infty$ in $L_p(\R)$, increasing sequences of integers $(M_j)_{j=0}^\infty$ and $(N_j)_{j=0}^\infty$, and a Haar type system $(g_j)_{j=0}^\infty$ in $L_p([0,1])$ such that $M_0=N_0=0$, $z_0=g_0=1_{[0,1]}$, and for all $n\in\N$ we have that
\begin{enumerate}
\item[(a)] $g_n\in span (h_j)_{j=M_{n-1}+1}^{M_{n}}$ and $(g_j)_{j=0}^n$ is an initial segment of a Haar type system.
\item[(b)] $span (z_j|_{[0,1]})_{j=0}^{N_n}\subseteq span (h_j)_{j=0}^{M_{n}}$ and each of the functions $(z_j|_{[0,1]^c})_{j=0}^{N_{n-1}}$ have disjoint support from each of the functions $(z_j|_{[0,1]^c})_{j=N_{n-1}+1}^{N_{n}}$.
\item[(c)] If $P_{M_{n-1}}$ is the basis projection onto $span(h_j)_{j=0}^{M_{n-1}}$ then $\|P_{M_{n-1}} x\|\leq \vp_n\|x\|$ for all $x\in span (z_j)_{j=N_{n-1}+1}^{N_n}$.  
\item[(d)] $(z_j)_{j=N_{n-1}+1}^{N_n}$ is $(1+\vp)-$basic,
\item[(e)] $dist(g_n,span_{N_{n-1}<j\leq N_n}(z_j))<\vp_n.$
\end{enumerate}

Before proving that this is possible, we show that building such a sequence $(z_j)_{j=0}^\infty$ will prove our theorem.  By (e), the span of $(z_j)_{j=0}^\infty$ contains a perturbation of a Haar type system for $L_p([0,1])$ and hence $L_p([0,1])$ is isomorphic to a subspace of the closed  span of $(z_j)_{j=0}^\infty$.  We now show that $(z_j)_{j=0}^\infty$ is a basic sequence.  Let $x=\sum_{j=0}^\infty a_j z_j\in span(z_j)_{j=0}^\infty$ and let $N\in\N$.   We will prove that $\|\sum_{j=0}^\infty a_j z_j\|\geq \frac{1}{2(1+\vp)^2} \|\sum_{j=0}^N a_j z_j\|$.

We denote $x_0=a_0 z_0$  and $x_n=\sum_{j=N_{n-1}+1}^{N_n} a_j z_j$ for all $n\in\N$.  We denote $y_0=x_0$ and $y_n=x_n-P_{M_{n-1}}x_n$ for all $n\in\N$.  By (c), we have that $\|y_n-x_n\|\leq \vp_n \|x_n\|$. As the Haar sequence is $1$-basic, we have by (b) that $(y_n)_{n=0}^\infty$ is $1$-basic.  As $(x_n)_{n=0}^\infty$ is a perturbation of $(y_n)_{n=0}^\infty$, we have that $(x_n)_{n=0}^\infty$ is $(1+\vp)$-basic. Let $K\in\N\cup \{0\}$ such that $N_{K}<N\leq N_{K+1}$.
Thus, 
$$\|x\|\geq (1+\vp)^{-1}\|\sum_{n=0}^K x_n\|\textrm{ and } \|x\|\geq (1+\vp)^{-1}\| x_{K+1}\|
$$
By (d), we have that $\|x_{K+1}\|\geq (1+\vp)^{-1} \|\sum_{j=N_{K}+1}^{N} a_j z_j\|$.
Thus, we have that 
\begin{align*}
\Big\|\sum_{j=0}^\infty a_j z_j\Big\|&\geq (1+\vp)^{-1}\max\Big(\Big\|\sum_{n=0}^K x_n\Big\|,\|x_{K+1}\|\Big)\\    &\geq (1+\vp)^{-1}\max\Big(\Big\|\sum_{n=0}^K x_n\Big\|,(1+\vp)^{-1}\Big\|\sum_{j=N_{K}+1}^{N} a_j z_j\Big\|\Big)\\  
 &\geq 2^{-1}(1+\vp)^{-2}\Big\|\sum_{n=0}^K x_n+\sum_{j=N_{K}+1}^{N} a_j z_j\Big\|\\  
 &= 2^{-1}(1+\vp)^{-2}\Big\|\sum_{j=0}^{N} a_j z_j\Big\|
\end{align*}
This proves that $(z_j)_{j=0}^\infty$ is $2(1+\vp)^2$-basic.
 Thus all that remains is to construct $(z_j)_{j=0}^\infty$ and $(g_j)_{j=0}^\infty$ by induction.

For the  base case we take $z_0=g_0=1_{[0,1]}$, $M_0=N_0=0$, $M_{-1}=N_{-1}=-1$, and we formally define $P_{-1}=0$ as the projection onto the zero vector.  Thus all five conditions are trivially satisfied for $n=0$.  Now let $k\in\N_0$ and assume that $(g_m)_{m=0}^k$ and $(z_m)_{m=0}^{N_k}$ have been chosen to satisfy conditions (a),(b),(c),(d), and (e).  For each $ m\in\N$ we let $m=2^{n_m-1}+j_m$ where $n_m\in\N$ and  $0\leq j_m< 2^{n_m-1}$.  For $1\leq m\leq k$, we denote $E_{2j_m,n_m}\subseteq[0,1]$ to be the support of $g_{m}^+$ and  $E_{2j_m+1,n_m}\subseteq[0,1]$ to be the support of $g_{m}^-$.  Being an initial segment of a Haar type system, $E_{2j_m,n_m}\cup E_{2j_m+1,n_m}=E_{j_m,n_m-1}$ for $1\leq m\leq k$, and for the induction we must find an appropriate partition of $E_{j_{k+1},n_{k+1}-1}$. Note that  if $j_k+1<2^{n_k-1}$ then $j_{k+1}=j_k+1$ and $n_{k+1}=n_k$; if $j_k+1=2^{n_k-1}$  then $j_{k+1}=0$ and $n_{k+1}=n_k+1.$ 

As $(g_m)_{m=0}^{k}$ is contained in the span of the initial segment of the Haar basis $(h_j)_{j=0}^{M_{k}}$, we may partition $E_{j_{k+1},n_{k+1}-1}$  into two sets of equal measure $E_{2j_{k+1},n_{k+1}}$ and $E_{2j_{k+1}+1,n_{k+1}}$ such that both sets are a finite union of disjoint dyadic intervals and for all $x\in span(h_j)_{j=0}^{M_{k}}$, the distribution of $x|_{E_{2j_{k+1},n_{k+1}}}$ is the same as the distribution of $x|_{E_{2j_{k+1}+1,n_{k+1}}}$.  We let $g_{k+1}=2^{(n_{k+1}-1)/p}(1_{E_{2j_{k+1},n_{k+1}}}-1_{E_{2j_{k+1}+1,n_{k+1}}})$.
As the support of $g_{k+1}^+$ and the support of $g_{k+1}^-$ are both finite unions of disjoint dyadic intervals, we have that $g_{k+1}\in span(h_j)_{j=1}^\infty$.  Let $0\leq m\leq M_{k}$. 
As the distribution of $h_m|_{E_{2j_{k+1},n_{k+1}}}$ is the same as the distribution of $h_m|_{E_{2j_{k+1}+1,n_{k+1}}}$, we have that $h_m^*(g_{k+1})=0$. Thus, $g_{k+1}\in span(h_j)_{j=M_{k}+1}^\infty$.

Thus, we have the following three properties.
\begin{enumerate}
\item[($\alpha$)] $(g_j)_{j=0}^{k+1}$ is the initial segment of a Haar type system in $L_p([0,1])$,
\item[($\beta$)] $g_{k+1}\in span(h_j)_{j=M_{k}+1}^\infty$,
\item[($\gamma$)] For all $x\in span(h_j)_{j=0}^{M_{k}}$, the distribution of $x|_{supp(g_{k+1}^+)}$ is the same as the distribution of $x|_{supp(g_{k+1}^-)}$.
\end{enumerate}

 By Lemma \ref{L:conditional_p} there exists $N\in\N$ and $(x_j)_{j=1}^{2N}$ in $\ell_p(\Z_{2N}\oplus\Z_{2N})$ such that 
\begin{enumerate}
\item $(x_j)_{j=1}^{2N}$ is ($1+\vp$)-basic.
\item If $f=(0)_{j=1}^{2N}\oplus((2N)^{-1/q})_{j=1}^{2N} \in \ell_q(\Z_{2N}\oplus\Z_{2N})$ then $(2N)^{1/q}|f(x)|\leq \frac{\vp_{k+1}}{M_k+1}\|x\|$  for all $x$ in the span of $(x_j)_{j=1}^{2N}$.
\item The distance from $(0)_{j=1}^{2N}\oplus((-1)^j (2N)^{-1/p})$ to the span of $(x_j)_{j=1}^{2N}$ is at most $\vp_{k+1}$.
\end{enumerate}

As in the proof of Theorem \ref{T:basisL2}, there exists a sequence of finite unions of disjoint dyadic intervals $(G_{j})_{j=1}^{2N}$ in $[0,1]$ such that 
\begin{enumerate}
\item[(i)] The sequence $(G_{j})_{j=1}^{2N}$ is pairwise disjoint and $\lambda(G_j)=\lambda(G_i)$ for all $i,j$.
\item[(ii)] $\cup_{j=1}^N G_{2j-1}$ is the support of $g_{k+1}^+$ and $\cup_{j=1}^N G_{2j}$ is the support of $g_{k+1}^-.$  
\item[(iii)] For all $x\in span (h_j)_{j=0}^{M_{k}}$, the sequence of functions $(x|_{G_{j}})_{j=1}^{2N}$ all have the same distribution.
\end{enumerate}

 Let $(H_j)_{j=1}^{2N}$ be a sequence of unit length intervals in $\R \setminus [0,1]$ with pairwise disjoint support which is disjoint from the support of $z_j$ for all $0\leq j\leq N_k$.
 We now define a map $\Psi:\ell_p(\Z_{2N}\oplus\Z_{2N})\rightarrow L_p(\R)$ by 
 $$\Psi(a_1,...,a_{2N},b_1,...,b_{2N})=\sum_{j=1}^{N} (2N)^{1/p}b_{2j-1} 1_{G_{2j-1}}g_{k+1}^+ +\sum_{j=1}^{N} (2N)^{1/p}b_{2j} 1_{ G_{2j}}g_{k+1}^-+\sum_{j=1}^{2N} a_j 1_{H_{j}}
 $$
By (i), (ii), and that $\|g_{k+1}^-\|=\|g_{k+1}^+\|=2^{-1/p}$ we have that $\|1_{G_{2j-1}}g_{k+1}^+\|=(2N)^{-1/p}$ and $\|1_{G_{2j}}g_{k+1}^-\|=(2N)^{-1/p}$ for all $1\leq j\leq N$.  Thus, $\Psi$ is an isometric embedding and maps positive elements of $\ell_p(\Z_{2N}\oplus\Z_{2N})$ to positive functions in $L_p(\R)$.  We let $N_{k+1}=N_k+2N$ and let $z_{N_k+j}=\Psi(x_j)$ for all $1\leq j\leq 2N$.  Thus, (d) is clearly satisfied.

Note that $\Psi((0,...,0)\oplus(\frac{1}{(2N)^{1/p}},\frac{-1}{(2N)^{1/p}},...,\frac{1}{(2N)^{1/p}},\frac{-1}{(2N)^{1/p}}))=g_{k+1}$, thus by (3) the distance from $g_{k+1}$ to the span of $(z_j)_{j=N_k+1}^{N_{k+1}}$ is at most $\vp_{k+1}$ which proves (e).


Let $z\in span(z_j)_{j=N_{k}+1}^{N_{k+1}}$ with $\|z\|=1$.  We now prove that $\|P_{M_k} z\|\leq\vp_{k+1}$. 
Note that $P_{M_k}(z)=\sum_{j=0}^{M_k} h_j^*(z) h_j$.
Let $1\leq m\leq M_k$.  We have that the functions $(h_m|_{G_j})_{j=1}^{2N}$ all have equal distribution and $g_{k+1}\in span(h_j)_{j=M_{k}+1}^\infty$. 
Hence, $h_m^*(1_{G_j})$ is independent of $j$.   Let $x=(a_1,...,a_{2N},b_1,...,b_{2N})\in span(x_j)_{j=1}^{2N}$ such that $\Psi(x)=z$. Let $f=(0)_{j=1}^{2N}\oplus((2N)^{-1/q})_{j=1}^{2N} \in \ell_q(\Z_{2N}\oplus\Z_{2N})$.  By (2), we have that $(2N)^{1/q}|f(x)|\leq \frac{\vp_{k+1}}{M_k+1}$.  Since the  biorthogonal functionals $(h_j^*)_{j=0}^\infty$ form the standard Haar basis in $L_q([0,1])$, $h_m^*$ is a multiple of $h_m$, and we denote this multiple by $C_{p,m}$.
We now have that 
\begin{align*}
    |h_m^*(z)|&=C_{p,m}|\int_0^1 h_m z \, dt|\\
    &=C_{p,m}\Big|\int_0^1 h_m\Psi(x) dt\Big|\\
    &=C_{p,m}\Big|\int_0^1 h_m\sum_{j=1}^{N} (2N)^{1/p}b_{2j-1} 1_{G_{2j-1}}g_{k+1}^+ +\sum_{j=1}^{N} (2N)^{1/p}b_{2j} 1_{ G_{2j}}g_{k+1}^- dt\Big|\\
        &=
        \Big|\sum_{j=1}^{N} b_{2j-1} +\sum_{j=1}^{N} b_{2j} \Big|(2N)^{1/p}2^{(n_{k+1}-1)/p}\Big|h_m^*(1_{G_1})\Big|\\
        &\leq (2N)^{1/q}|f(x)|\leq \frac{\vp_{k+1}}{M_k+1}
\end{align*}
Thus we have that $\|P_{M_k} z\|=\|\sum_{j=0}^{M_k} h_j^*(z) h_j\|\leq \sum_{j=0}^{M_k}\|h_j^*(z) h_j\|\leq\vp_{k+1}$.  This proves (c).  For all $1\leq j\leq 2N$, we have that  $G_j$ is a finite union of disjoint dyadic intervals.  Thus, $span (z_j|_{[0,1]})_{j=N_k+1}^{N_{k+1}}\subseteq span (h_j)_{j=0}^{\infty}$.  By ($\beta$), we also have that $g_{k+1}\in span (h_j)_{j=M_k+1}^{\infty}$. We now choose $M_{k+1}\in\N$ such that $span (z_j|_{[0,1]})_{j=N_k+1}^{N_{k+1}}\subseteq span (h_j)_{j=0}^{M_{k+1}}$ and $g_{k+1}\in span (h_j)_{j=0}^{M_{k+1}}$.  Thus, (a) holds and our proof is complete.

\end{proof}

 \section{Schauder frames}\label{S:frame}
  
 Previously, we have considered Schauder bases for Banach spaces, which give unique representations for vectors.  Given a Banach space $X$ with dual $X^*$, a sequence of pairs $(x_j,f_j)_{j=1}^\infty$ in $X\times X^*$ is called a {\em Schauder frame} or {\em quasi-basis} of $X$ if 
  \begin{equation}\label{E:frame}
x=\sum_{j=1}^\infty f_j(x) x_j\qquad\textrm{ for all }x\in X.  
\end{equation}
A Schauder frame is a possibly redundant coordinate system in that the sequence of coefficients $(f_j(x))_{j=1}^\infty$ which can be used to reconstruct $x$ in \eqref{E:frame} may not be unique.  Note that if $(x_j)_{j=1}^\infty$ is a Schauder basis of $X$ with biorthogonal functionals $(x_j^*)_{j=1}^\infty$ then $(x_j,x_j^*)_{j=1}^\infty$ is a Schauder frame of $X$.  Thus, Schauder frames are a generalization of Schauder bases.

For all $1\leq p<\infty$, there does not exist an unconditional Schauder frame $(x_j,f_j)_{j=1}^\infty$ for $L_p(\R)$ such that $(x_j)_{j=1}^\infty$ is a sequence of non-negative functions.  However, for all $1\leq p<\infty$, there does  exist a conditional Schauder frame $(x_j,f_j)_{j=1}^\infty$ for $L_p(\R)$ such that $(x_j)_{j=1}^\infty$ is a sequence of non-negative functions \cite{PS}. Indeed, if $(e_j)_{j=1}^\infty$ is a Schauder basis for $L_p(\R)$ with biorthogonal functionals $(e_j^*)_{j=1}^\infty$ then we may define a Schauder frame $(x_j, f_j)_{j=1}^\infty$ for $L_p(\R)$ by $x_{2j}=e_j^+$, $x_{2j-1}=e_j^-$, $f_{2j}=e_j^*$, and $f_{2j-1}=-e_j^*$ for all $j\in\N$.  

For each $1\leq p<\infty$ and $\lambda\in\R$, we may define the right translation operator $T_\lambda:L_p(\R)\rightarrow L_p(\R)$ by $T_\lambda f(t)=f(t-\lambda)$.  Given $1\leq p<\infty$, $f\in L_p(\R)$, and $(\lambda_j)_{j=1}^\infty\subseteq \R$, there have been many interesting results on the possible structure of $(T_{\lambda_j}f)_{j=1}^\infty$, and the relation on the values $(\lambda_j)_{j=1}^\infty$ can be very subtle.  For example, if $1\leq p\leq 2$ then a simple fourier transform argument gives that $(T_{j} f)_{j\in\Z}$ does not have dense span in $L_p(\R)$ \cite{AO}.  However, if $\vp_j\neq0$ for all $j\in\Z$ and $\vp_j\rightarrow 0$ for $|j|\rightarrow \infty$ then there does exist $f\in L_2(\R)$ such that $(T_{j+\vp_j} f)_{j\in\Z}$ has dense span in $L_2(\R)$ \cite{O}.  For any $(\lambda_j)_{j=1}^\infty\subseteq \R$, $1\leq p<\infty$, and $f\in L_p(\R)$ the sequence $(T_{\lambda_j} f)_{j=1}^\infty$ is not an unconditional Schauder basis for $L_p(\R)$ (\cite{OZ} for $p=2$, \cite{OSSZ} for $1<p\leq 4$, and \cite{FOSZ} for $4< p$).  However, if $2<p$ and $(\lambda_j)_{j=1}^\infty$ is unbounded then there exists $f\in L_p(\R)$ and a sequence of functionals $(g_j)_{j=1}^\infty$ such that $(T_{\lambda_j}f, g_j)_{j=1}^\infty$ is an unconditional Schauder frame of $L_p(\R)$.  It was not known for $1\leq p<2$  if there exists $(\lambda_j)_{j=1}^\infty\subseteq \R$, $f\in L_p(\R)$, and a sequence of functionals $(g_j)_{j=1}^\infty$ such that $(T_{\lambda_j} f, g_j)_{j=1}^\infty$ is an unconditional Schauder frame or even conditional Schauder frame  for $L_p(\R)$.  If the sequence $(g_j)_{j=1}^\infty$ is semi-normalized (in particular  $(\|g_j\|^{-1})_{j=1}^\infty$ is bounded) then $(T_{\lambda_j} f, g_j)_{j=1}^\infty$ cannot be an unconditional Schauder frame for $L_p(\R)$ for $1\leq p\leq 2$ \cite{BC}.

We will prove that for all $1\leq p<\infty$ that there exists a single non-negative function $f\in L_p(\R)$ such that $(T_{\lambda_j} f, f_j)_{j=1}^\infty$ is a Schauder frame for $L_p(\R)$ for some sequence of constants $(\lambda_j)_{j=1}^\infty$ and some sequence of functionals $(f_j)_{j=1}^\infty$.  We will obtain this as a corollary from the following general result about the existence of  certain Schauder frames, which we believe to be of independent interest.
The proof of the following theorem is inspired by Pelczynski's proof that every separable Banach space with the bounded approximation property is isomorphic to a complemented subspace of a Banach space with a Schauder basis \cite{P}.

\begin{thm} \label{T:1}
Let $X$ be a Banach space with a Schauder basis $(e_j)_{j=1}^\infty$.  Suppose that $D\subseteq X$ is a subset whose span is dense in $X$.  Then there exists a Schauder frame (quasi-basis) for $X$ whose vectors are elements of $D$.
\end{thm}
\begin{proof}
As  $(e_j)_{j=1}^\infty$ is a Schauder basis of $X$,  there exists $\vp_j\searrow0$ such that if $(u_j)_{j=1}^\infty\subseteq X$ and $\|e_j-u_j\|<\vp_j$ for all $j\in\N$ then $(u_j)_{j=1}^\infty$ is a Schauder basis of $X$.  As the span of $D$ is dense in $X$ we may choose $(u_j)_{j=1}^\infty\subseteq span(D)$ such that $\|e_j-u_j\|<\vp_j$ for all $j\in\N$.  Let $(u_j^*)_{j=1}^\infty$ be the sequence of biorthogonal functionals to $(u_j)_{j=1}^\infty$.  For each $n\in\N$, we  may choose a linearly independent and finite ordered set $(x_{j,n})_{j\in J_n}$ in $D$ such that $u_n$ can be expressed as the finite sum $u_n=\sum_{j\in J_n}a_{j,n}x_{j,n}$  where $a_{j,n}$ are non-zero scalars.

Let $C_n$ be the basis constant of $(x_{j,n})_{j\in J_n}$ and choose $N_n\in\N$ such that $C_n\leq N_n$.   We currently have that $u_n$ may be uniquely expressed as $u_n=\sum_{j\in J_n}a_{j,n}x_{j,n}$, but to make a Schauder frame we will use the redundant expansion $u_n=\sum_{i=1}^{N_n} \sum_{j\in J_n} N_n^{-1} a_{j,n}x_{j,n}$.  
 We claim that $((x_{j,n},N_n^{-1} a_{j,n} u_n^*))_{n\in\N,1\leq i \leq N_n,j\in J_n}$ is a Schauder frame of $X$ where we order $\{(n,i,j)\}_{n\in\N,1\leq i \leq N_n,j\in J_n}$ lexicographically.  That is, $(n_1,i_1,j_1)\leq (n_2,i_2,j_2)$ if and only if 
\begin{enumerate}
\item $n_1<n_2$, or
\item $n_1=n_2$ and $i_1<i_2$, or
\item $n_1=n_2$ and $i_1=i_2$ and $j_1\leq j_2$.
\end{enumerate}

Let $x\in X$ and $\vp>0$.  Choose $N\in\N$ such that $\|\sum_{n=m_1}^{m_2} u_n^*(x) u_n\|<\vp $ for all $m_2\geq m_1\geq N$.   Consider a fixed $(n_0,i_0,j_0)$ with $n_0> N$, $j_0\in J_{n_0}$, and $1\leq i_0\leq N_{n_0}$.  We now have that,
\begin{align*}
&\Big\|x-\sum_{(n,i,j)\leq (n_0,i_0,j_0)} N_n^{-1} a_{j,n} u_n^*(x) x_{j,n}\Big\|\\
 &\leq \Big\|x-\sum_{n=1}^{n_0-1} \sum_{i=1}^{N_n}  \sum_{j\in J_n}N_n^{-1} a_{j,n} u_n^*(x) x_{j,n}\Big\|+\Big\|\sum_{i=1}^{i_0-1}\sum_{j\in J_{n_0}}N_{n_0}^{-1} a_{j,n_0} u_{n_0}^*(x) x_{j,n_0}\Big\|+\Big\|\sum_{j=1}^{j_0}  N_{n_0}^{-1} a_{j,n_0} u_{n_0}^*(x) x_{j,n_0}\Big\|\\
 &\leq \Big\|x-\sum_{n=1}^{n_0-1} u_n^*(x) u_{n}\Big\|+\sum_{i=1}^{i_0-1} N_{n_0}^{-1}\|u_{n_0}^*(x) u_{n_0}\Big\|+\Big\|\sum_{j=1}^{j_0}  N_{n_0}^{-1} a_{j,n_0} u_{n_0}^*(x) x_{j,n_0}\Big\|\\
 &< \vp+\vp+C_{n_0}\Big\|\sum_{j\in J_{n_0}}  N_{n_0}^{-1} a_{j,n_0} u_{n_0}^*(x) x_{j,n_0}\Big\|\\
 &= \vp+\vp+C_{n_0}N_{n_0}^{-1} \|u_{n_0}^*(x) u_{n_0}\|\\
 &< \vp+\vp+\vp\hspace{5cm}\textrm{ as }C_{n_0}\leq N_{n_0}.\\
\end{align*}
We have that $\sum_{(n,i,j)} N_n^{-1} a_{j,n} u_n^*(x) x_{j,n}$ converges to $x$, and hence $((x_{j,n},N_n^{-1} a_{j,n} u_n^*))_{n\in\N,1\leq i \leq N_n,j\in J_n}$ is a Schauder frame of $X$.

\end{proof}

The previous theorem applied to Banach spaces with a Schauder basis, and we now show that the same conclusion can be obtained for separable Banach spaces with the bounded approximation property.

\begin{cor}\label{C:pos_frame}
Let $X$ be a separable Banach space with the bounded approximation property (i.e. $X$ has a quasi-basis).  Suppose that $D\subseteq X$ is a subset whose span is dense in $X$.  Then there exists a Schauder frame (quasi-basis) for $X$ whose vectors are elements of $D$.
\end{cor}
\begin{proof}
As $X$ is separable and has the bounded approximation property there exists a Banach space $Y$ with a basis such that $X\subseteq Y$ and there is a bounded projection $P:Y\rightarrow X$.  As the span of $D$ is dense in $X$, the span of $D\cup (I_Y-P)Y$ is dense in $Y$, where $I_Y$ is the identity operator on $Y$.  By Theorem \ref{T:1}, there exists a Schauder frame $(x_j, f_j)_{j=1}^\infty\cup(y_j,g_j)_{j=1}^\infty$ for $Y$, where $x_j\in D$ and $y_j\in  (I_Y-P)Y$ for all $j\in\N$.  The projection of a Schauder frame onto a complemented subspace is a Schauder frame for that subspace.  Thus,  $(Px_j, f_j|_{X})_{j=1}^\infty\cup(Py_j,g_j|_X)_{j=1}^\infty$ is a Schauder frame for $X$.  This is the same as, $(x_j, f_j|_X)_{j=1}^\infty\cup(0,g_j|_X)_{j=1}^\infty$.  Hence,  $(x_j, f_j|_X)_{j=1}^\infty$ is a Schauder frame of $X$ whose vectors are in $D$.
\end{proof}

We now give the following application to translations of a single positive vector.

\begin{cor}\label{C:trans_pos}
For all $1\leq p<\infty$, the Banach space $L_p(\R)$ has a Schauder frame of the form  $(x_j, f_j)_{j=1}^\infty$ where $(x_j)_{j=1}^\infty$ is a sequence of translates of a single non-negative function.  In the range $1<p<\infty$ this function can be taken to be the indicator function of a bounded interval in $\R$, and for $p=1$ the function can be any non-negative function whose Fourier transform has no real zeroes.
\end{cor}

\begin{proof}
We first consider the case $p=1$.  Let $f\in L_1(\R)$. By Wiener's tauberian theorem, the set of translations of $f$ has dense span in $L_1(\R)$ if and only if the Fourier transform of $f$ has no real zeroes \cite{W}.  Thus by Theorem \ref{T:1} if the Fourier transform of $f$ has no real zeroes then there exists a sequence of translations $(x_j)_{j=1}^\infty$ of $f$ and a sequence of linear functionals $(f_j)_{j=1}^\infty$ such that $(x_j, f_j)_{j=1}^\infty$ is a Schauder frame of $L_1(\R)$.  As an example of a function $f\in L_1(\R)$ such that $\hat{f}$ has no real zeroes, one can take $f(t)=e^{-t^2}$ for all $t\in\R$.

We now fix $1<p<\infty$ and consider the interval $(0,1]\subseteq\R$.  Note that the span of the indicator functions of bounded intervals in $\R$ is dense in $L_p(\R)$.  Thus we just need to prove that every indicator function of a bounded interval is in the closed span of the translates of $(0,1]$ and then apply Theorem \ref{T:1} to get a Schauder frame of translates of the indicator function of $(0,1]$.  Let $D\subseteq L_p(\R)$ be the span of the set of translates of $1_{(0,1]}$.  

Let $1>\vp>0$.  For each $\lambda\in\R$, we denote $T_\lambda:L_p(\R)\rightarrow L_p(\R)$ to be the operator which shifts functions $\lambda$ to the right.  That is, for all $f\in L_p(\R)$, $T_\lambda f(t)=f(t-\lambda)$ for all $t\in\R$.   Let $x_1= 1_{(0,1]}-T_\vp1_{(0,1]}=1_{(0,\vp]}-1_{(1,1+\vp]}$.  Thus, $x_1\in D$.  For $n\in\N$, we define $x_{n+1}\in D$ by 
$$x_{n+1}=\sum_{j=0}^{n} T_{j}x_1=\sum_{j=0}^n 1_{(j,j+\vp]}-1_{(j+1,j+1+\vp]}=1_{(0,\vp]}-1_{(n+1,n+1+\vp]}.$$  

As $1<p<\infty$, the sequence $(x_n)_{n=1}^\infty$ converges weakly to $1_{(0,\vp]}$.  Thus, $1_{(0,\vp]}$ is in the weak-closure and hence norm-closure of $D$ as $D$ is convex.  This proves that every indicator function of an interval of length at most $1$ is contained in $\overline{D}$.  As every bounded interval is the disjoint union of finitely many intervals of length at most $1$, we have that the indicator function of any bounded interval is contained in $\overline{D}$.
\end{proof}

When using a Schauder basis or Schauder frame to reconstruct a vector in a Banach space, we have that the partial sums of the series in \eqref{E:basis} and \eqref{E:frame} converge in norm.  A Banach lattice is a Banach space endowed with an appropriate partial order.  For example $L_p(\R)$ is a Banach lattice with the partial order given for $f,g\in L_p(\R)$ by  $f\leq g$ if and only if $f(t)\leq g(t)$ for a.e. $t\in \R$.  When considering Banach lattices, one cares about both the norm structure of the Banach space as well as the endowed order structure.  This leads us to consider Schauder bases and Schauder frames where the partial sums of the reconstruction formula converge in order as well as in norm.

Let $(y_n)_{n=1}^\infty$ be a sequence in a Banach lattice $X$.  We say that $(y_n)_{n=1}^\infty$ {\em converges uniformly} to $y$ and write $y_n\xrightarrow{u} y$ if there exists a positive vector $w\in X$ such that for all $\vp>0$ there exists $N\in\N$ such that $|y-y_n|\leq \vp w$ for all $n\geq N$. The vector $w$ is called a {\em regulator} of the sequence $(y_n)_{n=1}^\infty$.  Let $(x_j)_{j=1}^\infty$ be a Schauder basis for a Banach lattice $X$ with biorthogonal functionals $(x_j^*)_{j=1}^\infty$.  We say that $(x_j)_{j=1}^\infty$ is a {\em bibasis} for $X$ if for all $x\in X$ we have that $\sum_{j=1}^n x_j^*(x) x_j\xrightarrow{u} x$.  Similarly, let $(x_j, f_j)_{j=1}^\infty$ be a Schauder frame for a Banach lattice $X$.   We say that $(x_j,f_j)_{j=1}^\infty$ is a {\em u-frame} for $X$ if for all $x\in X$ we have that $\sum_{j=1}^n f_j(x) x_j\xrightarrow{u} x$.  The difference between the two names (bibasis and u-frame) is that the bibasis condition is equivalent to multiple different properties \cite[Theorem 3.1]{TT} or \cite[Theorem 20.1]{T}, whereas this is not the case in the context of frames. 

We now next extend Theorem~\ref{T:1} to the setting of Banach lattices with a bibasis. 

\begin{thm} \label{T:1b}
Let $X$ be a Banach lattice with a bibasis $(e_j)_{j=1}^\infty$.  Suppose that $D\subseteq X$ is a subset whose span is dense in $X$.  Then there exists a u-frame for $X$ whose vectors are elements of $D$.
\end{thm}
\begin{proof}
The proof begins analogously to Theorem \ref{T:1}, noting that small perturbations of bibases are bibases (\cite[Theorem 4.2]{TT}). 

We construct $(u_n)_{n=1}^\infty$ and $((x_{j,n})_{j\in J_n})_{n=1}^\infty$ as in the proof of Theorem \ref{T:1}. We currently have that $((x_{j,n},N_n^{-1} a_{j,n} u_n^*))_{n\in\N,1\leq i \leq N_n,j\in J_n}$ is a Schauder frame of $X$ in the lexicographical order whenever the $N_n$ are sufficiently large. We now need to show that it is a u-frame.  For each $n\in\N$, we define $v_n=\sum_{j\in J_n}|x_{j,n}|$. Let $v=\sum_{n=1}^\infty \frac{1}{2^n}\frac{v_n}{\|v_n\|}$ and choose $N_n\in \mathbb{N}$ such that $N_n\geq 4^n\|u_n^*\|\|v_n\| \max_{j\in J_n}|a_{j,n}|$. Then for each $x\in X$ and each subset $I_n$ of $J_n$, $|\sum_{j\in I_n}  N_{n}^{-1} a_{j,n} u_{n}^*(x) x_{j,n}|\leq \frac{1}{2^{n}}v\|x\|$. We claim that $((x_{j,n},N_n^{-1} a_{j,n} u_n^*))_{n\in\N,1\leq i \leq N_n,j\in J_n}$ is a u-frame of $X$ where, again, we order $\{(n,i,j)\}_{n\in\N,1\leq i \leq N_n,j\in J_n}$ lexicographically.  

Let $x\in X$ and  let $w\in X^+$ be a regulator for $\sum_{j=1}^n u_n^*(x)u_n\xrightarrow{u}x$. In particular, for all $\vp>0$ there exists $N\in\N$ such that $\frac{1}{2^N}\|x\|<\varepsilon$ and $|\sum_{n=m_1}^{m_2} u_n^*(x) u_n|\leq \vp w$ for all $m_2\geq m_1\geq N$.   Consider a fixed $(n_0,i_0,j_0)$ with $n_0> N$, $j_0\in J_{n_0}$, and $1\leq i_0\leq N_{n_0}$.  By analogous estimates one shows that 
\begin{align*}
&\left| x-\sum_{(n,i,j)\leq (n_0,i_0,j_0)} N_n^{-1} a_{j,n} u_n^*(x) x_{j,n}\right| \leq 3\varepsilon (v\vee w).
\end{align*}

Hence $\sum_{(n,i,j)} N_n^{-1} a_{j,n} u_n^*(x) x_{j,n}$ converges to $x$ uniformly with regulator $v\vee w$, proving that $((x_{j,n},N_n^{-1} a_{j,n} u_n^*))_{n\in\N,1\leq i \leq N_n,j\in J_n}$ is a u-frame of $X$.
\end{proof}

The Haar system is not a bibasis for $L_1(\R)$, but the Haar system is a bibasis for $L_p(\R)$  for the range $1<p<\infty$ \cite{GKP}.  Thus, the following corollary follows from Theorem \ref{T:1b} and Corollary \ref{C:trans_pos}.

\begin{cor}\label{C:trans_pos_u}
For all $1< p<\infty$, the Banach space $L_p(\R)$ has a u-frame of the form  $(x_j, f_j)_{j=1}^\infty$ where $(x_j)_{j=1}^\infty$ is a sequence of translates of a single non-negative function.  Furthermore, this function can be taken to be the indicator function of a bounded interval in $\R$.
\end{cor}

\section{Open problems}

Johnson and Schechtman constructed a Schauder basis for  $L_1(\R)$ consisting of non-negative functions \cite{JS}, and in Theorem \ref{T:basisL2} we construct a Schauder basis for  $L_2(\R)$ consisting of non-negative functions. The following remaining cases are still open. 
\begin{problem}
Let $1<p<\infty$ with $p\neq 2$.  Does  $L_p(\R)$ have a Schauder basis  consisting of non-negative functions?
\end{problem}

In Theorem \ref{T:basisLp}, we showed that $L_p(\R)$ contains a basic sequence $(f_n)_{n=1}^\infty$ of non-negative functions such that $L_p(\R)$ embeds into the closed span of $(f_n)_{n=1}^\infty$.  
Furthermore, the proof gives that for all $\vp>0$,  $(f_n)_{n=1}^\infty$ can be chosen to be $(2+\vp)$-basic.  

\begin{problem}
Let $1\leq p<\infty$ with $p\neq 2$.  For all $\vp>0$, does  $L_p(\R)$ contain a $(1+\vp)$-basic sequence $(f_n)_{n=1}^\infty$ of non-negative functions such that $L_p(\R)$ embeds into the closed span of $(f_n)_{n=1}^\infty$? What is the infimum of the set of all basis constants of  non-negative bases in $L_1(\R)$?
\end{problem}

The questions about non-negative bases in $L_p(\R)$ that are considered here and in \cite{PS} naturally extend to general Banach lattices. We say that a Schauder basis $(x_n)_{n=1}^\infty$ of a Banach lattice is {\em positive} if $x_n\geq 0$ for all $n\in\N$.  We say that a Schauder basis $(x_n)_{n=1}^\infty$ has {\em positive biorthogonal functionals} if the biothorgonal functionals $(x_n^*)_{n=1}^\infty$ satisfy $x_n^*\geq 0$ for all $n\in\N$. 
In the case of $L_p(\mu)$ or $C([0,1])$,  Schauder bases of non-negative functions correspond exactly with Schauder bases of positive vectors.  The unit vector basis for $\ell_p$ is a positive Schauder basis for all $1\leq p<\infty$, and the Faber-Schauder system in $C([0,1])$ is a Schauder basis of non-negative functions \cite{F}.

The existence of positive bases in $L_1$ has the following application to the general theory of Banach lattices:
\begin{prop}
Every separable Banach lattice embeds lattice isometrically into a Banach lattice with a positive Schauder basis.
\end{prop}
\begin{proof}
It was shown in \cite{LLOT} that every separable Banach lattice embeds lattice isometrically into $C(\Delta,L_1)$, where $\Delta$ denotes the Cantor set and $C(\Delta,L_1)$ denotes the Banach space of continuous functions from $\Delta$ to $L_1$. Hence, it suffices to show that $C(\Delta,L_1)$ has a positive Schauder basis.\\
By \cite{JS}, $L_1$ has a basis $(f_j)$ of positive vectors, and by the proof of \cite[Proposition 2.5.1]{Sem}, $C(\Delta)$ has a basis $(d_i)$ of positive vectors. For each $i,j\in \N$, define $d_i \otimes f_j\in C(\Delta,L_1)$ via $(d_i \otimes f_j)(t)=d_i(t)f_j$ for all $t\in\Delta$. Clearly, $d_i\otimes f_j\geq 0$ in $C(\Delta,L_1)$.

Now note that $C(\Delta,L_1)$ is lattice isometric to $C(\Delta)\otimes_{\lambda} L_1$, the injective tensor product of $C(\Delta)$ and $L_1$.  We order the collection $(d_i\otimes f_j)_{i,j\in\N}$ into the sequence $(z_k)_{k=1}^\infty$ by $z_1=d_1\otimes f_1$ and for $k>1$ we let
\[z_k= \begin{cases} 
      d_i\otimes f_{n+1} \ \ \text{for}\  k=n^2+i\textrm{ where $i,n\in\N$ and }1\leq i\leq n+1, \\
      \
      d_{n+1}\otimes f_{n+1-i} \ \ \text{for}\  k=n^2+n+1+i\textrm{ where $i,n\in\N$ and }1\leq i\leq n. 
   \end{cases}
\]

Then \cite[Theorem 18.1 and Corollary 18.3]{Sin} guarantee that $(z_k)_{k=1}^\infty$ is a Schauder basis.
\end{proof}
We have given several examples of Banach lattices with positive bases ($L_1(\R)$, $L_2(\R),  C([0,1])$, $\ell_p$, $C(\Delta,L_1)$, etc.) By duality it is easy to see that $L_2(\R)$ has a basis with positive biorthogonal functionals, and using \cite[Proposition 10.1, p.~321]{Sin} one sees that if $K$ is compact, Hausdorff and $C(K)$ is infinite-dimensional then $C(K)$ cannot have a basis with positive biorthogonal functionals. Obviously, the spaces $\ell_p$ have a basis with positive biorthogonal functionals whenever $1\leq p<\infty$. A general question to pose is:

\begin{problem}
Give further examples of Banach lattices possessing positive bases and/or bases with positive biorthogonal functionals. Of particular interest are Banach lattices possessing bases but lacking positive bases.
\end{problem}

There are other weaker forms of coordinate systems for which one can impose positivity conditions. For example, we refer the reader to \cite[Remark 7.13]{TT} for questions regarding the structure of Banach lattices possessing FDDs with positivity properties on their associated projections.
Recall that a {\em Markushevich basis} of a Banach space $X$ is a biorthogonal system ($x_n, x_n^*)_{n=1}^\infty$ such that the closed span of $(x_n)_{n=1}^\infty$ is $X$ and the collection of  functionals $(x_n^*)_{n=1}^\infty$ separates the points of $X$.
Obviously, when $X$ is a Banach lattice one can put positivity conditions on $x_n$ and $x_n^*$, and in \cite{PS} it is shown that for all $1\leq p<\infty$, $L_p(\R)$ has a Markushevich basis consisting of non-negative functions. This leaves another general question:
\begin{problem}
Which separable Banach lattices have Markushevich bases consisting of positive vectors?  Which separable Banach lattices have Markushevich bases consisting of positive functionals?
\end{problem}

Suppose that $X$ is a Banach lattice with a Schauder frame $(x_j, f_j)_{j=1}^\infty$.  By splitting up each vector into its positive and negative parts, we obtain that the sequence of pairs $(x_1^+,f_1), (x_1^-,-f_1),(x_2^+,f_2),(x_2^-,-f_2),...$ is a Schauder frame of $X$ consisting of positive vectors.  Thus, every Banach lattice with a Schauder frame has a Schauder frame with positive vectors.  Similarly, every Banach lattice with a Schauder frame has a Schauder frame with positive functionals.
On the other hand, in \cite{PS} it is proven for all $1\leq p<\infty$ that $L_p(\R)$ does not have an unconditional Schauder frame  consisting of positive vectors. 
\begin{problem}\label{Q:lattice}
Which separable Banach lattices have an unconditional Schauder frame with positive vectors?  Which separable Banach lattices have an unconditional Schauder frame with positive functionals?
\end{problem}


\end{document}